\documentclass[aop,MSNbibl,nameyear,dvips]{arximspdf}
\usepackage{graphicx}
\doi{10.1214/12-AOP823} 
\volume{43}
\issue{1}
\pubyear{2015}
\firstpage{1}
\lastpage{43}

\makeatletter
\newcommand{\rrvert}{\vert}
\newcommand{\llvert}{\vert}
\newcommand{\Z}{\mathbb{Z}}
\newcommand{\C}{\mathbb{C}}
\newcommand{\R}{\mathbb{R}}
\newcommand{\Hb}{\mathbb{Hb}}
\newcommand{\E}{\mathbb{E}}
\renewcommand{\i}{\mathrm{i}}
\newcommand{\si}{\sigma}
\newcommand{\te}{\theta}
\newcommand{\be}{\beta}
\newcommand{\de}{\delta}
\newcommand{\Pc}{\mathbf{P}}
\newcommand{\Pl}{\mathcal{P}}
\newcommand{\D}{\mathcal{D}}
\newcommand{\om}{{{\mathsf{w}}}}
\newcommand{\omb}{{\overline{\om}}}
\newcommand{\spt}{\mathsf{s}}
\newcommand{\rpt}{\mathsf{r}}
\newcommand{\Kast}{\mathsf{Kast}}
\newcommand{\X}{\mathsf{X}}
\newcommand{\Pp}{\mathbb{P}}
\newcommand{\Ga}{\mathfrak{C}}
\newcommand{\hor}{\mathsf{H}}
\newcommand{\ver}{\mathsf{V}}
\newcommand{\GFF}{\mathsf{GFF}}
\newcommand{\G}{\mathcal{G}}
\newcommand{\wt}{\mbox{
\includegraphics{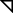}
}}
\newcommand{\bt}{\mbox{
\includegraphics{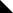}
}}
\newcommand{\lozv}{\mbox{
\includegraphics{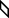}
}}
\newcommand{\lozl}{\mbox{
\includegraphics{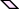}
}}
\newcommand{\lozs}{\mbox{
\includegraphics{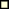}
}}

\newtheorem{proposition}{Proposition}[section]
\newtheorem{lemma}[proposition]{Lemma}
\newtheorem{corollary}[proposition]{Corollary}
\newtheorem{theorem}[proposition]{Theorem}
\newproclaim{definition}[proposition]{Definition}
\newproclaim{remark}[proposition]{Remark}
\makeatother

\begin{document}
\begin{frontmatter}

\title{Asymptotics of uniformly random lozenge tilings of polygons. Gaussian free field\thanksref{T1}}
\runtitle{Asymptotics of random lozenge tilings. Gaussian free field\hspace*{5pt}}

\begin{aug}
\author[A]{\fnms{Leonid} \snm{Petrov}\corref{}\ead[label=e1]{l.petrov@neu.edu}}
\runauthor{L. Petrov}
\affiliation{Northeastern University and Institute for Information
Transmission~Problems}
\address[A]{Department of Mathematics\\
Northeastern University\\
360 Huntington ave.\\
Boston, Massachusetts 02115\\
USA\\
and\\
Institute for Information\\
Transmission Problems\\
Bolshoy Karetny per. 19\\
Moscow, 127994\\
Russia\\
\printead{e1}}
\end{aug}
\thankstext{T1}{Supported in part by the RFBR-CNRS Grants 10-01-93114 and 11-01-93105.}

\received{\smonth{9} \syear{2012}}
\revised{\smonth{11} \syear{2012}}

%
\begin{abstract}
We study large-scale height fluctuations of random stepped surfaces
corresponding to uniformly random lozenge tilings of polygons on the
triangular lattice. For a class of polygons (which allows arbitrarily
large number of sides), we show that these fluctuations are
asymptotically governed by a Gaussian free (massless) field. This
complements the similar result obtained in Kenyon [\textit{Comm. Math. Phys.} \textbf{281} (2008) 675--709]
about tilings of regions without
frozen facets of the limit shape.

In our asymptotic analysis we use the explicit double contour integral
formula for the determinantal correlation kernel of the model obtained
previously in Petrov [Asymptotics of random lozenge tilings via Gelfand--Tsetlin
schemes (2012) Preprint].
\end{abstract}
%

%

%

\end{frontmatter}

\section{Introduction and main result} 
\label{secintroductionandmainresult}

We begin with a description of the model and formulation of necessary
previous results which motivate the main result of the present paper.
The latter is stated in Section~\ref{subresults} below.

\subsection{Model of uniformly random tilings} 
\label{submodelofuniformlyrandomtilings}

Consider a polygon drawn on the regular triangular lattice as shown in
Figure~\ref{figpolygontrianglattice}.

In the present paper we study the model of uniformly random tilings of
such polygons by lozenges ($={}$rhombi) of three types:\vspace*{6pt}

\begin{center}
\includegraphics{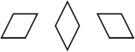}
\end{center}
\vspace*{6pt}

An example of such a tiling is presented in Figure~\ref{figGFFtiling}, left.
Equivalent formulations of the model include:
\begin{itemize}
\item (Dimer interpretation.) Lozenge tilings of a polygon are in a
bijective correspondence with dimer coverings ($={}$perfect matchings) on
the part of the honeycomb graph located inside the polygon; see
Figure~\ref{figGFFtiling}, right.
\item (Stepped surfaces.) One can view each tiling such as in
Figure~\ref{figGFFtiling}, left, as a 2-dimensional projection of a stepped
surface, that is, of a continuous 3-dimensional surface glued out of
$1\times1\times1$ boxes with sides parallel to three coordinate lines
in space.
\end{itemize}

%
\begin{figure}

\includegraphics{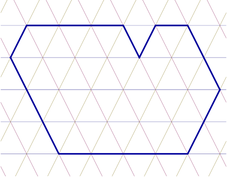}

\caption{Polygon on the triangular lattice.} \label
{figpolygontrianglattice}
\end{figure}
%
\begin{figure}

\includegraphics{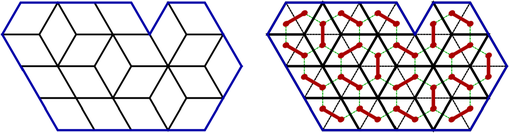}

\caption{Tiling of a polygon on the regular triangular lattice (left)
and its dimer interpretation (right).}
\label{figGFFtiling}
\end{figure}

The model of uniformly random lozenge tilings of polygons has received
significant attention over the past years: \citet{CohnKenyonPropp2000},
\citet{OkounkovKenyon2007Limit}, \citet{Kenyon2004Height}. See also
\citet
{Kenyon2007Lecture} for a detailed exposition of the subject and more
references.


\subsection{Affine transform and the class of polygons} 
\label{subaffinetransformandtheclassofpolygons}

For technical convenience, we perform a simple affine transform of
lozenges which were present in Figure~\ref{figGFFtiling}; see
Figure~\ref{figaffinetransform}.
%
\begin{figure}[b]

\includegraphics{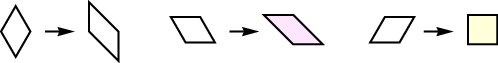}

\caption{Affine transform of lozenges.}
\label{figaffinetransform}
\end{figure}
After this transform, tilings of polygons will look like the one in
Figure~\ref{figpolygontilingparticles} below. Polygons which are
tiled will thus be drawn on the standard square grid, with all sides
parallel either to one of the coordinate axes, or the vector $(-1,1)$.
We will denote the horizontal and the vertical integer coordinates on
the square grid by $x$ and $n$, respectively.

\begin{figure}

\includegraphics{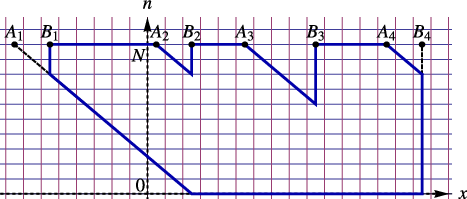}

\caption{A polygon from the class we consider. In this example $k=4$,
and the polygon has $3k=12$ sides.}
\label{figpolygonalregion}
\end{figure}

We will restrict ourselves to polygons of a special kind, as shown in
Figure~\ref{figpolygonalregion}. Every polygon $\Pc$ we consider can
be parametrized by two integers $N=1,2,\ldots$, and $k=2,3,\ldots$ (the
polygon has $3k$ sides), and by $2k$ (proper) half-integers
\[
A_1<B_1<A_2<B_2<
\cdots<A_k<B_k,\qquad A_i,B_i\in
\Z':=\Z+\tfrac12,
\]
subject to the condition $\sum_{i=1}^{k}(B_i-A_i)=N$ which ensures that
there is at least one lozenge tiling of $\Pc$. The bottom side of $\Pc$
lies on the horizontal axis $n=0$, and all the $k-1$ top sides [the
$i$th such side has endpoints $(B_i,N)$ and $(A_{i+1},N)$, $i=1,\ldots,k-1$] lie on one and the same line $n=N$.


\subsection{Height function of a tiling} 
\label{subheightfunction}

We will view our lozenge tilings (as on Figure~\ref{figpolygontilingparticles} below) as projections of 3-dimensional
stepped surfaces onto the $(x,n)$ plane. The surface itself can thus be
interpreted as a graph of a function $h(x,n)$ which is called the \emph
{height function} of the tiling. To be concrete, let us stick to the
convention that lozenges of type \lozl\ correspond to horizontal
planes, that is, planes where the height function is constant. We also
require that $h(x,n)$ is zero near the lower left corner of the
polygon. See Figure~\ref{figGFFheight} for an example and Section~\ref{subdefinitionoftheheightfunction} below for a precise definition.
See also [\citet{Kenyon2007Lecture}, Section~2.8] for more discussion.
%
\begin{figure}[t]

\includegraphics{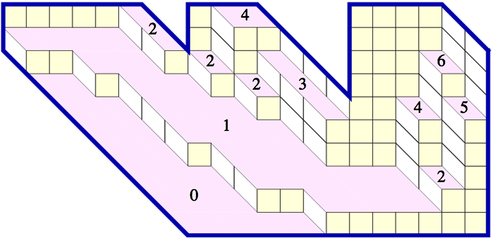}

\caption{Values of the height function $h(x,n)$ on each horizontal plateau.}
\label{figGFFheight}
\end{figure}

The main object of the present paper is the height function
corresponding to the uniformly random lozenge tiling of a polygon $\Pc
$. We assume that $\Pc$ belongs to the class of polygons described in
Section~\ref{subaffinetransformandtheclassofpolygons}. We denote
this \emph{random height function} by $h_{\Pc}(x,n)$.


\subsection{Limit shape} 
\label{sublimitshape}

We consider large $N$ asymptotics of random tilings as all dimensions
of the polygon $\Pc=\Pc(N)$ grow. That is, let the parameters
$A_i(N),B_i(N)$ of $\Pc(N)$ behave as
%
\begin{equation}
\label{scaleAiBi} A_i(N)=[a_i N]+\mathrm{const},\qquad
B_i(N)=[b_i N]+\mathrm{const} \qquad\mbox{($i=1,\ldots,k$)}.
\end{equation}
Here $a_1<b_1<\cdots<a_k<b_k$ are new continuous parameters which
satisfy $\sum_{i=1}^{k}(b_i-a_i)=1$. The constants above are bounded
uniformly in $N$ and are needed to ensure that $A_i(N),B_i(N)\in\Z'$
and $\sum_{k=1}^{N}(B_i(N)-A_i(N))=N$.

In \citet{CohnKenyonPropp2000} it was shown that in this $N\to\infty$
regime, the rescaled random stepped surface concentrates around a
nonrandom limit shape which can be obtained as a unique solution to a
suitable variational problem; see also \citet{CohnLarsenPropp}, \citet
{DMB1997Lozenge} and \citet{Destainville1998Lozenge}. This solution was
described in \citet{OkounkovKenyon2007Limit} by means of the complex
Burgers equation. For polygons we consider in the present paper, the
limit shape is an algebraic surface.

More precisely, the limit shape result means that the height function
$h_{\Pc(N)}$ obeys the following law of large numbers (with almost sure
convergence):
%
\begin{equation}
\label{LLN} \frac{h_{\Pc(N)}([\chi N],[\eta N])}{N}\to\boldsymbol{h}(\chi,\eta),\qquad N\to\infty,
\end{equation}
where $(\chi,\eta)$ are the new global continuous coordinates, and
$\boldsymbol{h}(\chi,\eta)$ is the function whose graph is the limit
shape. The new coordinates $(\chi,\eta)$ are assumed to belong to the
limiting polygon $\Pl$ which is parametrized by $\{a_i,b_i\}_{i=1}^{k}$
in the same way as it was for $\Pc(N)$ and $\{A_i(N),B_i(N)\}
_{i=1}^{k}$ in Section~\ref{subaffinetransformandtheclassofpolygons}; see Figure~\ref{figfrozenboundary}. The new polygon $\Pl$ is located inside the
strip $0\le\eta\le1$.

%
\begin{figure}[t]

\includegraphics{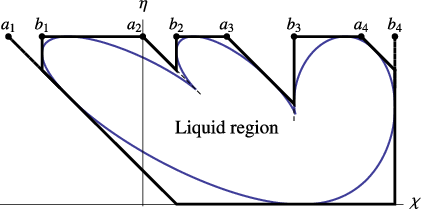}

\caption{The limiting polygon $\Pl$ on the $(\chi,\eta)$ plane and the
frozen boundary curve.}
\label{figfrozenboundary}
\end{figure}

A feature of the model we deal with is that the limit shape develops
\emph{frozen facets} where the function $\boldsymbol{h}(\chi,\eta)$ is
linear. In other words, frozen facets correspond to zones inside $\Pl$
where lozenges of only one type are asymptotically present. Along with
the frozen facets, there is a connected open \emph{liquid region} $\D
\subset\Pl$. For $(\chi,\eta)\in\D$, the limiting height function
$\boldsymbol h(\chi,\eta)$ is curved: asymptotically inside the liquid
region one sees a random mixture of all types of lozenges; for example,
see \citeauthor{CohnLarsenPropp} [(\citeyear{CohnLarsenPropp}), Figure~2],
\citeauthor{BorodinGorin2008} [(\citeyear{BorodinGorin2008}), Figure~5],
\citeauthor{OkounkovKenyon2007Limit} [(\citeyear{OkounkovKenyon2007Limit}), Figure~1] for illustrations of uniformly
random tilings with small mesh where the limit shape and the frozen
boundary are clearly seen.

\begin{remark}\label{rmkbulkreference}
There are more precise results in this direction. Namely, the
asymptotic local distribution of lozenges around a given global
position $(\chi,\eta)\in\D$ is governed by an ergodic translation
invariant Gibbs measure on tilings of the whole plane. Such a measure
is unique up to fixed proportions of lozenges of all types \citet
{Sheffield2008}, and these proportions depend on the slope of the limit
shape at the given point $(\chi,\eta)$. For polygons in the class
described in Section~\ref{subaffinetransformandtheclassofpolygons}, this result was
established in \citet{Petrov2012}. For $k=2$ (when the polygon is a
hexagon) this was obtained earlier in \citet{BKMM2003}, \citet
{Gorin2007Hexagon}. For tilings of regions when the limit shape has no
frozen parts, the same local behavior was shown in \citet
{Kenyon2004Height}. See also \citet{okounkov2003correlation} and \citet
{KOS2006} for more detail on the limiting translation invariant ergodic
Gibbs measures.
\end{remark}


\subsection{Complex structure on the limit shape surface} 
\label{subcomplexstructureonthelimitshapesurface}

Complex coordinate on the limit shape surface was introduced in \citet
{OkounkovKenyon2007Limit} [see also \citet{Kenyon2004Height}], and by a
different technique in \citet{Petrov2012}. From the results of \citet
{Petrov2012} it follows (see Section~\ref{subcriticalpointsoftheaction} below) that there exists a
diffeomorphism $\om=\om(\chi,\eta)$ from the liquid region $\D$ to the
upper half plane $\Hb:=\{z\in\C\dvtx\Im z>0\}$. [In \citet{Petrov2012}
the complex coordinate $\om$ was denoted by $\om_c$.] The function
$\om
(\chi,\eta)$ is algebraic, and it satisfies the following degree $k$ equation:
%
\begin{equation}
\label{omegaalgebraicequation} (\om-\chi)\prod_{i=1}^{k}(
\om-a_i)= (\om-\chi+1-\eta)\prod_{i=1}^{k}(
\om-b_i),
\end{equation}
and a version of the differential complex Burgers equation [see also
\citet{OkounkovKenyon2007Limit}],
%
\begin{equation}
\label{omegacomplexBurgersequation} \frac{\om(\chi,\eta)-\chi}{1-\eta}\cdot \frac{\partial\om(\chi,\eta)}{\partial\chi}=-
\frac{\partial\om(\chi,\eta)}{\partial\eta}.
\end{equation}
The complex coordinate $\om(\chi,\eta)$ can be used, in particular, to
describe the local asymptotics of random tilings mentioned in
Remark \ref{rmkbulkreference}; see \citet{Petrov2012}, Sections~2.3--2.4 for more detail. The complex structure $\om(\chi,\eta)$ on the
liquid region $\D$ is employed in our description of asymptotics of
fluctuations of the height function; see Section~\ref{subresults} below.


\subsection{Gaussian free field} 
\label{subgaussianfreefield}

Before we proceed to describing our results, let us first briefly
discuss the object which governs the asymptotics of fluctuations of the
height function, namely, the Gaussian free field. This subsection is
adapted from \citet{Sheffield2007GFF}, see that survey for a detailed
and systematic discussion.

The Gaussian free (massless) field $\GFF$ on the upper half plane $\Hb$
is a probability Gaussian measure supported on a suitable class of
generalized functions on $\Hb$ (and not on ordinary functions). In
particular, the value $\GFF(z)$ at a point $z\in\Hb$ does not make sense.

The distribution of $\GFF$ may be understood as follows. For any
sequence\break of compactly supported smooth test functions $\{\phi_r\}
_{r=1}^{\infty}$, the pairings\break $\{\GFF(\phi_r)\}_{r=1}^{\infty}$
form a
sequence of mean zero Gaussian random variables with covariances
\[
 \E\bigl(\GFF(\phi_k)\GFF(\phi_l)\bigr)
=\int_{\Hb\times\Hb} |dz_1|^{2}|dz_2|^{2}
\phi_k(z_1) \phi_l(z_2)
\G(z_1,z_2).
\]
Here $(\cdot,\cdot)$ in the first integral is the usual inner
product, and
%
\begin{equation}
\label{Greenfunction} \G(z,w):=-\frac{1}{2\pi}\ln\biggl\llvert
\frac{z-w}{z-\bar w} \biggr\rrvert,\qquad z,w\in\Hb
\end{equation}
is the Green function for the Laplace operator on the upper half plane
$\Hb$ with Dirichlet boundary conditions.

Even though the value of $\GFF$ at a point cannot be defined, one can
still think that the expectations of products of values of $\GFF$ at
pairwise distinct points $z_1,\ldots,z_\spt$ are well defined and are
given by
\[
\E\bigl(\GFF(z_1)\cdots\GFF(z_{\spt})\bigr)= %
\cases{\displaystyle\sum_{\si} \prod
_{i=1}^{\spt/2} \G (z_{\si(2i-1)}, z_{\si(2i)} ),
&\quad $\mbox{$\spt$ even}$;\vspace*{2pt}
\cr
0,&\quad $\mbox{$\spt$ odd}$, } %
\]
with sum over all fixed point free involutions ($={}$pairings) $\si$ on
$\{
1,\ldots,\spt\}$. Indeed, for a finite number of test functions,
%
\begin{equation}
\label{GFFfunctionsmoments} \E\bigl(\GFF(\phi_1)\cdots\GFF(
\phi_\spt)\bigr)= \int_{\Hb^{\spt}} \E\bigl(
\GFF(z_1)\cdots\GFF(z_{\spt})\bigr) \prod
_{i=1}^{\spt}|dz_i|^{2}
\phi_i(z_i).
\end{equation}
The moments (\ref{GFFfunctionsmoments}) uniquely determine the
Gaussian free field.


\subsection{Results} 
\label{subresults}

Now we are in a position to describe the main results of the present
paper. We are interested in asymptotics of fluctuations
%
\begin{equation}
\label{HNdef} H_N(\chi,\eta):=h_{\Pc(N)}\bigl([\chi N],[\eta
N]\bigr)-\E h_{\Pc(N)}\bigl([\chi N],[\eta N]\bigr)
\end{equation}
of the height function (of a uniformly random tiling) around its mean.

\begin{theorem}[(Moment convergence of fluctuations to $\GFF$)]
\label{thmmomentconvergenceintro}
For pairwise distinct points $(\chi_1,\eta_1),\ldots,(\chi_\spt,\eta
_\spt)$ inside the liquid region $\D$, as we scale the polygon $\Pc(N)$
as in (\ref{scaleAiBi}), the following convergence of moments holds:
%
\begin{eqnarray}\label{GFFcorrelations}
&&
\lim_{N\to\infty} \pi^{\spt/2}\E \bigl(
H_N(\chi_1,\eta_1 )\cdots
H_N(\chi_\spt,\eta_\spt) \bigr)
\nonumber\\
&&\qquad= \E \bigl(\GFF\bigl(\om(\chi_1,\eta_1)\bigr)\cdots
\GFF\bigl(\om(\chi_\spt,\eta _\spt)\bigr) \bigr)
\\
&&\qquad= %
\cases{\displaystyle \sum_{\si} \prod
_{i=1}^{\spt/2} \G \bigl(\om(\chi_{\si(2i-1)},
\eta_{\si(2i-1)}), \om(\chi_{\si(2i)},\eta_{\si(2i)}) \bigr), &\quad $
\mbox{$\spt$ even}$;\vspace*{2pt}
\cr
0,&\quad $\mbox{$\spt$ odd}$,} %
\nonumber
\end{eqnarray}
where the sum is taken over all fixed point free involutions $\si$ on
$\{1,\ldots,\spt\}$.
\end{theorem}

\begin{theorem}[(Central limit theorem for fluctuations of the height function)]
\label{thmweakconvergenceintro}
The random function $\sqrt\pi H_N(\chi,\eta)$ on $\D$ weakly converges
as $N\to\infty$ to the $\om$-pullback of the Gaussian free field
$\GFF$
on $\Hb$.
\end{theorem}

Theorem \ref{thmweakconvergenceintro} means that for any smooth
compactly supported test function $\phi$ on~$\D$, we have the weak
convergence as $N\to\infty$,
%
\begin{eqnarray} \label{weakconvergenceofonefunction} \sqrt\pi\int_{\D}
\phi(\chi,\eta)H_N(\chi,\eta)\,d\chi \,d\eta &\to &\int
_{\D}\phi(\chi,\eta)\GFF\bigl(\om(\chi,\eta)\bigr)\,d\chi \,d\eta
\nonumber
\\[-8pt]
\\[-8pt]
\nonumber
&=&\int_{\Hb} \phi\bigl(\om^{-1}(z)\bigr) J(z)
\GFF(z) |dz|^{2},
\end{eqnarray}
where $J(z)$ is the Jacobian of the change of variables $z\to(\chi,\eta
)$ by $\om^{-1}$ [which in fact can be explicitly calculated using
(\ref
{omegaalgebraicequation})--(\ref{omegacomplexBurgersequation})].
Theorem \ref{thmweakconvergenceintro} follows from Theorem~\ref
{thmmomentconvergenceintro} plus an additional bound on moments of
fluctuations of the height function at infinitesimally close points;
see Section~\ref{subconvergencetogffproofoftheoremthmweakconvergenceintro}.

It is worth noting that Gaussian free field fluctuations in random
tiling models were also obtained [along with \citet{Kenyon2004Height}]
in \citet{Ferrari2008}, \citet{Duits2011GFF} and \citet{Kuan2011GFF}.


\subsection{Strategy of the proof and organization of the paper} %
\label{substrategyoftheproofandorganizationofthepaper}

Our proof is based on an explicit formula for the determinantal
correlation kernel of uniformly random tilings which was established in
\citet{Petrov2012}. We recall these results in Section~\ref{secdeterminantalstructureofrandomtilings}. In Section~\ref{secheightfunctionanditsmultipointfluctuations} we write the
multipoint fluctuations $\E (H_N(\chi_1,\eta_1 )\cdots H_N(\chi
_\spt,\eta_\spt) )$ of the height function in terms of that
correlation kernel. This allows us to establish Theorem \ref
{thmmomentconvergenceintro} and then Theorem \ref
{thmweakconvergenceintro} in Section~\ref{seccompletingtheproofs}
using certain fine asymptotic properties of the correlation kernel
which are obtained in Section~\ref{secasymptoticsofthekernel}.

Our argument generally follows the approach of \citet{Ferrari2008}
(especially see Section~5 in that paper) which in turn was partly
inspired by \citet{Kenyon2004Height}. We also use some ideas from \citet
{Duits2011GFF}. However, our correlation kernel has a more complicated
structure than those of \citet{Ferrari2008} and \citet{Duits2011GFF}: the
critical points of the action (Section~\ref{subcriticalpointsoftheaction}) which are solutions of~(\ref
{omegaalgebraicequation}) cannot be determined explicitly. Thus, in
our Section~\ref{secasymptoticsofthekernel} in order to investigate
asymptotics of the kernel, we must employ certain new considerations
(Sections~\ref{submovingthecontours}--\ref
{subestimatingthesingleintegral}).



\section{Determinantal structure of random tilings} 
\label{secdeterminantalstructureofrandomtilings}

In this section we recall the formula of \citet{Petrov2012} for the
determinantal correlation kernel of our model of uniformly random
tilings. Then we extend that kernel and describe the joint distribution
of all three types of lozenges in our random tiling. Except for Section~\ref{subextensionofkandjointdistributionofthreetypesoflozenges},
this section is essentially taken from \citet{Petrov2012}.

\subsection{Interlacing particle arrays} 
\label{subinterlacingparticlearrays}

Let $\Pc$ be a polygon of our class in the $(x,n)$ plane (see Section~\ref{subaffinetransformandtheclassofpolygons}) contained inside
the horizontal strip $0\le n\le N$. We pass from tilings of $\Pc$ to
interlacing particle arrays as follows. We first trivially extend any
tiling of $\Pc$ to a tiling of the whole strip $0\le n\le N$ with $N$
small triangles added on top; see Figure~\ref{figpolygontilingparticles}.
%
\begin{figure}[b]

\includegraphics{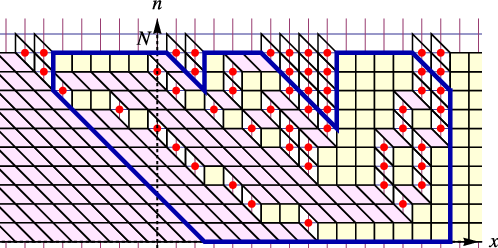}

\caption{Tiling of a polygon and the corresponding interlacing
particle array.}
\label{figpolygontilingparticles}
\end{figure}
Then we place a particle in the center of every lozenge of type \lozv.

Thus, we get a particle array $\X:=\{\mathsf{x}_{j}^{m}\dvtx
m=1,\ldots,N;j=1,\ldots,m\}\in\Z^{N(N+1)/2}$ with precisely $m$ particles at each
$m$th horizontal level, $m=0,1,\ldots, N$. Because we started from a
tiling of $\Pc$, these particles must satisfy the \emph{interlacing
constraints}
%
\begin{equation}
\label{interlacing} \mathsf{x}_{j+1}^{m}<
\mathsf{x}_{j}^{m-1}\le\mathsf{x}_{j}^{m}
\end{equation}
(for all $j$'s and $m$'s for which these inequalities can be written
out). Moreover, the particles in the top row of the array are fixed,
%
\begin{eqnarray}
\label{fixedtoprow}&& \bigl\{\mathsf{x}_{N}^{N}<\cdots<
\mathsf{x}_{1}^{N}\bigr\}\nonumber\\
&&\qquad=\bigl\{A_{1}+\tfrac
12<A_{1}+\tfrac 32<\cdots<B_1-\tfrac32<B_1-
\tfrac12
\\
&&\hspace*{17pt}\qquad <A_2+\tfrac12<\cdots <B_2-\tfrac12<\cdots<A_{k}+
\tfrac12<\cdots<B_k-\tfrac12\bigr\}\nonumber
\end{eqnarray}
(here $\{A_i,B_i\}_{i=1}^{k}$ are the parameters of $\Pc$; see Section~\ref{subaffinetransformandtheclassofpolygons}). Clearly, lozenge
tilings of $\Pc$ and such interlacing arrays $\X$ with fixed top row
(\ref{fixedtoprow}) are in a bijective correspondence. For a
connection of these arrays with Gelfand--Tsetlin schemes (an object
related to branching of representations of unitary groups), see, for
example, \citet{Petrov2012}, Section~3.


\subsection{Determinantal correlation kernel} 
\label{subdeterminantalcorrelationkernel}

We see from Section~\ref{subinterlacingparticlearrays} that the
uniform measure on the set of all tilings of the polygon $\Pc$ is the
same as the uniform measure on the space of interlacing integer arrays
$\X=\{\mathsf{x}_{j}^{m}\}$ with fixed top row~(\ref
{fixedtoprow}). We denote
both measures by $\Pp_{\Pc}$. Viewing $\X$ as a particle configuration,
we can also think of the measure $\Pp_\Pc$ as of a point process on
$\Z
\times\{1,2,\ldots,N\}$.

\begin{definition}\label{defcorrelationfunctions}
Let $(x_1,n_1),\ldots,(x_\spt,n_\spt)$ be pairwise distinct positions,
$x_i\in\Z$, $1\le n_i\le N$. The \emph{correlation functions} of the
point process $\Pp_\Pc$ are defined as
\begin{eqnarray*}
&& \rho_\spt(x_1,n_1;\ldots;x_\spt,n_\spt)
\\
&& :=
\Pp_\Pc \bigl(\mbox {there is a particle of the random configuration
$\bigl\{\mathsf{x}_j^{m}\bigr\}$}
\\
&&\hspace*{92pt}\mbox{at position $(x_i,n_i)$ for every $i=1,
\ldots,\spt $} \bigr).
\end{eqnarray*}
\end{definition}

It is well known that the measure $\Pp_{\Pc}$ on interlacing particle
arrays is \emph{determinantal} (for instance, this fact can be deduced
from the Kasteleyn theory; see Section~\ref{subinversekasteleynmatrix} below). That is, there exists a function
$K(x,n;y,m)$ (the \emph{correlation kernel}), such that
%
\begin{equation}
\label{correlationkerneldef} \rho_\spt(x_1,n_1;
\ldots;x_\spt,n_\spt)= \det\bigl[K(x_i,n_i;x_j,n_j)
\bigr]_{i,j=1}^{\spt}
\end{equation}
for any $\spt$ and any collection of pairwise distinct positions
$(x_1,n_1),\ldots,(x_\spt,n_\spt)$. About determinantal point processes
in general see the surveys \citet{Soshnikov2000}, \citet
{peres2006determinantal}, \citet{Borodin2009}.

In \citet{Petrov2012} the following explicit formula for the correlation
kernel $K$ of random interlacing arrays $\X$ with fixed top row (\ref
{fixedtoprow}) was obtained:

%
\begin{theorem}[{[\citet{Petrov2012}]}]
\label{thmK}
For $1\le n_1\le N$, $1\le n_2\le N-1$ and $x_1,x_2\in\Z$, the
correlation kernel of the point process $\Pp_\Pc$ has the
form\setcounter{footnote}{1}\footnote{Here and below $1_{\{\cdot\cdot\cdot\}}$ denotes the indicator of a
set, and $(y)_m:=y(y+1)\cdots(y+m-1)$, $m=1,2,\ldots$ [with $(y)_0:=1$]
is the Pochhammer symbol.}
%
\begin{eqnarray}
\label{Kformula} &&K(x_1,n_1;x_2,n_2)\nonumber\\
&&\qquad=
-1_{n_2<n_1}1_{x_2\le x_1}\frac{(x_1-x_2+1)_{n_1-n_2-1}}{(n_1-n_2-1)!} +\frac{(N-n_1)!}{(N-n_2-1)!}
\nonumber
\\[-8pt]
\\[-8pt]
\nonumber
&&\qquad\quad{} \times \frac1{(2\pi\i)^{2}} \oint_{\Ga(x_2)}\,dz
\oint_{\mathfrak{c}(\infty)}\,dw \frac{(z-x_2+1)_{N-n_2-1}}{(w-x_1)_{N-n_1+1}} \frac{1}{w-z} \\
&&\qquad\quad{} \times\prod
_{i=1}^{k}
\frac{(A_i+1/2-w)_{B_i-A_i}}{(A_i+1/2-z)_{B_i-A_i}}.\nonumber
\end{eqnarray}
The contours in $z$ and $w$ are positively (counter-clockwise) oriented
and do not intersect. The contour $\Ga(x_2)$ in $z$ encircles the
integer points $x_2,x_2+1,\ldots,B_k-\frac12$ and only them (i.e.,
does not contain $x_2-1,x_2-2,\ldots$ and $B_k+\frac12,B_k+\frac
32,\ldots$). The contour $\mathfrak{c}(\infty)$ in $w$ contains $\Ga
(x_2)$ and
all the points $x_1,x_1-1,\ldots,x_1-(N-n_1)$.
\end{theorem}
The above explicit formula for the kernel $K$ is our main tool in the
present paper.


\subsection{Inverse Kasteleyn matrix} 
\label{subinversekasteleynmatrix}

Here let us recall the connection [\citet{Petrov2012}, Section~6]
between the above kernel $K$ (Theorem \ref{thmK}) and the Kasteleyn
matrix of the honeycomb graph $G_{\Pc}$ inside our polygon $\Pc$
(Figure~\ref{figGFFtiling}, right). We will use it to write down the
joint distribution of three types of lozenges \lozv, \lozs\ and
\lozl\ in Section~\ref{subextensionofkandjointdistributionofthreetypesoflozenges} below.

The honeycomb graph $G_\Pc$ is bipartite; its vertices correspond to
two types of (triangle) faces in the dual triangular lattice,\vspace*{6pt}

\begin{center}
\includegraphics{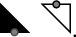}
\end{center}\vspace*{6pt}

We will encode each such triangle by the position $(x,n)$ of the
mid-point of its horizontal side. The Kasteleyn matrix of the graph
$G_{\Pc}$ is its adjacency matrix with rows and columns parametrized by
white and black triangles, respectively; for example, see \citet
{Kenyon2007Lecture}. Inside the polygon, this matrix looks as
%
\begin{eqnarray}
\label{Kasteleynmatrix} \Kast\bigl(\wt(x,n);\bt(y,m)\bigr)= %
\cases{
1,&\quad$\mbox{if $(y,m)=(x,n)$}$;\vspace*{2pt}
\cr
1,&\quad$\mbox{if $(y,m)=(x,n-1)$}$;
\vspace*{2pt}
\cr
1,&\quad$\mbox{if $(y,m)=(x+1,n-1)$}$;\vspace*{2pt}
\cr
0,&\quad$
\mbox{otherwise}$;} %
\end{eqnarray}
see Figure~\ref{figlozengestriangles}.
%
\begin{figure}

\includegraphics{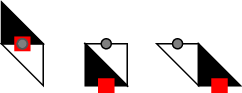}

\caption{Edges of three directions in the graph $G_\Pc$ encoded by
pairs of triangles.}
\label{figlozengestriangles}
\end{figure}
For $\wt(x,n)$ on the boundary of the graph $G_\Pc$, the $\wt(x,n)$th
row of $\Kast$ will contain less than three ones, and the same for the
$\bt(y,m)$th column.

It is known that the determinant $\det[\Kast(\wt(x,n);\bt(y,m)]$, where
$\wt(x,n)$ and $\bt(y,m)$ run over all possible white and black
triangles in $G_{\Pc}$, is equal to the total number of lozenge tilings
of the polygon $\Pc$. As [\citet{Kenyon2007Lecture}, Corollary~3]
suggests, $\Kast^{-1}$ can serve as a correlation kernel for the
uniform measure on tilings of $\Pc$. A more precise statement is as follows:

\begin{theorem}[{[\citet{Petrov2012}]}]
\label{thmKasteleyn}
The inverse Kasteleyn matrix and the correlation kernel $K$ of Theorem
\ref{thmK} are related as follows [for all possible values of $(x,n)$
and $(y,m)$]:
\[
\Kast^{-1}\bigl(\bt(y,m);\wt(x,n)\bigr)=(-1)^{y-x+m-n}K(x,n;y,m).
\]
\end{theorem}


\subsection{Extension of $K$ and joint distribution of three types of
lozenges} 
\label{subextensionofkandjointdistributionofthreetypesoflozenges}

Here we compute probabilities that a random tiling has lozenges of
prescribed types at prescribed positions (i.e., the joint distribution
of all types of lozenges). These probabilities are given by
determinants similar to the correlation functions of interlacing
particle arrays (Definition \ref{defcorrelationfunctions}), but with
an extended kernel.

Let, by agreement, the position of every lozenge be encoded by the
position of its white triangle (i.e., by the position of the circle dot
on Figure~\ref{figlozengestriangles}). We introduce the following
extended kernel (here and below $\te_j\in\{\lozv,\lozs,\lozl\}$ are
types of lozenges):
%
\begin{equation}
\label{extendedkernel} \qquad K_\te(x_1,n_1,
\te_1;x_2,n_2,\te_2):=
\cases{K(x_1,n_1;x_2,n_2),&\quad $
\mbox{if $\te_2=\lozv$}$;\vspace*{2pt}
\cr
-K(x_1,n_1;x_2,n_2-1),&\quad $
\mbox{if $\te_2=\lozs$}$;\vspace*{2pt}
\cr
K(x_1,n_1;x_2+1,n_2-1),&\quad $
\mbox{if $\te_2=\lozl$}$.}\hspace*{-5pt} %
\end{equation}

%
\begin{proposition}\label{propKte}
For any collection of lozenges of types $\te_1,\ldots,\te_\spt\in\{
\lozv,\lozs,\lozl\}$ at (pairwise distinct) positions $(x_i,n_i)$,
$i=1,\ldots,\spt$, we have
\begin{eqnarray*}
&&\Pp_{\Pc} \bigl(\mbox{There is a lozenge of type  $
\te_r$ at $(x_r,n_r)$ for all $r=1,
\ldots,\spt$} \bigr)
\\
& &\qquad =\det\bigl[K_\te(x_i,n_i,\te
_i;x_j,n_j,\te_j)
\bigr]_{i,j=1}^{\spt}.
\end{eqnarray*}
\end{proposition}

\begin{pf}
This is a direct consequence of the above Theorem \ref{thmKasteleyn}
and also of \citet{Kenyon2007Lecture}, Corollary 3. Namely:
\begin{itemize}
\item there is a lozenge of type $\lozv$ at $(x,n)$ if and only if the
dimer covering contains the edge $ (\bt(x,n);\wt(x,n) )$;
\item there is a lozenge of type $\lozs$ at $(x,n)$ if and only if the
dimer covering contains the edge $ (\bt(x,n-1);\wt(x,n) )$;
\item there is a lozenge of type $\lozl$ at $(x,n)$ if and only if the
dimer covering contains the edge $ (\bt(x+1,n-1);\wt(x,n) )$.
\end{itemize}
Then, using \citet{Kenyon2007Lecture}, Corollary 3, we can write the
probability
\[
\Pp_\Pc \bigl(\mbox{There is a lozenge of type $\te_r$ at $(x_r,n_r)$ for all $r=1,
\ldots,\spt$} \bigr)
\]
as a determinant of a suitable matrix with entries $\Kast^{-1}$. In
this matrix there will be three types of rows (recall that they are
indexed by black triangles) corresponding to different types of
lozenges as above. Then, using the relation between $\Kast^{-1}$ and
the kernel $K$ (Theorem \ref{thmKasteleyn}), we complete the proof.
\end{pf}
A similar property for a different tiling model (of an infinite region)
was obtained in \citet{Ferrari2008}, Theorem 5.2; see also \citet
{borodin-gr2009q}, Section~7.2.


\section{Height function and its multipoint fluctuations} 
\label{secheightfunctionanditsmultipointfluctuations}

In this section we discuss the concept of a height function of a
tiling. In Section~\ref{subexpectationofaproductofhorizontalandverticalsums}, for our
model of uniformly random tilings, we express the multipoint moments of
fluctuations of the height function through the correlation kernel $K$
of Theorem~\ref{thmK}.

\subsection{Definition of the height function} 
\label{subdefinitionoftheheightfunction}


Let $\Pc$ be a polygon from our class (Section~\ref{subaffinetransformandtheclassofpolygons}). Fix a tiling of $\Pc
$. It is possible to define the \emph{height function} of this tiling
which at every position $(x,n)\in\Pc$ is equal to
%
\begin{equation}
\label{heightfunctiondefinition} \quad h(x,n):=\sum_{m\dvtx m\le n}
1\bigl\{\mbox{there is a lozenge of type $\lozv$ or $\lozs$ at $(x,m)$}\bigr\}.
\end{equation}
Clearly, this implies that the height function is constant on each
horizontal plateau consisting of lozenges of type $\lozl$; see
Figure~\ref{figGFFheight}.

With every tiling one can associate three families of nonintersecting
lattice paths; for example, see \citet{Petrov2012}, Section~2.5.
%
\begin{figure}[t]

\includegraphics{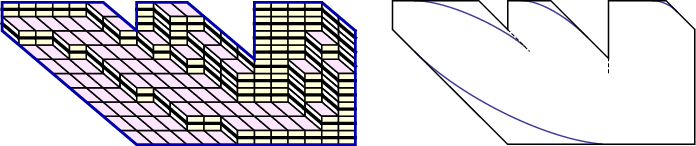}

\caption{``Level lines'' of the height function (left) and the
corresponding parts of the frozen boundary (right).}
\label{figpathHS}
\end{figure}
\begin{figure}[b]

\includegraphics{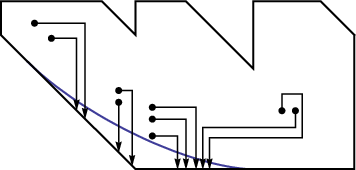}

\caption{Paths to the frozen boundary along which we calculate the
height function at $([\chi_iN],[\eta_iN])$, $i=1,\ldots,\spt$.}
\label{figGFFpaths}
\end{figure}
Nonintersecting paths in one of these families shown in Figure~\ref{figpathHS} (left) can serve as ``level lines'' of the height
function. Namely, $h(x,n)$ at a given point is equal to the number of
these nonintersecting paths lying between $(x,n)$ and the line $n=0$.
See also Figure~\ref{figpolygontilingparticles} where the tiling is
extended so that formula (\ref{heightfunctiondefinition}) and the
interpretation with the ``level lines'' make full sense.

\begin{remark}
The two other families of nonintersecting paths [\citet{Petrov2012},
Section~2.5] give two other possible ways to define the height
function; see also \citet{Kenyon2007Lecture}, Section~2.8 for a more
detailed discussion.
\end{remark}

\subsection{Paths to the boundary} 
\label{subpathstotheboundary}

Now let $\Pc(N)$ be a sequence of polygons scaled as explained in
Section~\ref{sublimitshape}. We would like to study asymptotics of
multipoint fluctuations of the random height function $h_{\Pc(N)}$
[corresponding to the uniformly random tiling of $\Pc(N)$] around
global positions $(\chi_1,\eta_1),\ldots,\break (\chi_\spt,\eta_{\spt
})\in\D$.
Here $\D$ is the liquid region (Figure~\ref{figfrozenboundary}).
Namely, we are interested in the asymptotic behavior of the
expectations entering (\ref{GFFcorrelations}),
%
\begin{equation}
\label{EHNpathstoboundary} \E \bigl( H_N(\chi_1,
\eta_1 )\cdots H_N(\chi_\spt,
\eta_\spt) \bigr),
\end{equation}
where $H_N(\chi_j,\eta_j)$'s are the fluctuations of the height
function defined in (\ref{HNdef}).

To compute values of the height function entering (\ref
{EHNpathstoboundary}), one could use formula~(\ref
{heightfunctiondefinition}). But then several indicators
corresponding to the same point $(x,m)$ in~(\ref
{heightfunctiondefinition}) will enter the resulting expression. This
will lead to certain technical complications which we can easily avoid.
Namely, let us choose $\spt$ piecewise-linear simple paths with
horizontal and vertical pieces such that each $i$th path connects
$([\chi_i N],[\eta_i N])$ and the lower left part of the frozen boundary
as in Figure~\ref{figGFFpaths}. We choose this part of $\partial\D$
because on the corresponding facet the height function will be
asymptotically equal to zero; cf. Figure~\ref{figGFFheight}. By
agreement, we assume that while crossing the frozen boundary, each path
goes in a vertical direction and proceeds vertically down until it hits
the boundary of the polygon. By taking $N$ larger if necessary,\vadjust{\goodbreak} we
require that these $\spt$ paths do not intersect.\footnote{Except for
the case when some of the points $(\chi_i,\eta_i)$ coincide; then we
still do not allow intersections away from the starting points $([\chi
_i N],[\eta_i N])$.} We also require the number of piecewise-linear
segments in each path to be bounded (uniformly in $N$).

Such paths can be constructed using the diffeomorphism $\om\dvtx\D
\to
\Hb$ (see Section~\ref{subcomplexstructureonthelimitshapesurface} and Section~\ref{subcriticalpointsoftheaction} below) which maps the frozen
boundary to the real line. It is not hard to show that in $\Hb$ the
desired (continuous) nonintersecting paths exist. Then in $\D$ we can
approximate the images of these paths under $\om^{-1}$ by
piecewise-linear paths. Since the number of points $(\chi_i,\eta_i)$
(and paths) is finite, we can also make sure that the number of
segments in these paths is bounded uniformly in $N$.

\begin{remark}
We could also use paths ending at any part of the frozen boundary in
Figure~\ref{figpathHS} (right) because on the corresponding facets
our height function asymptotically becomes constant, and we subtract
this constant in the definition of $H_N$ (\ref{HNdef}). However, we
use only paths as oi Figure~\ref{figGFFpaths} to simplify the
notation in Section~\ref{subestimatesofthekernelclosetotheedge} below.
\end{remark}

\begin{definition}
For integers $x<x'$ and $n$, denote
%
\begin{equation}
\label{horsum} \hor_{x,x'}(n):=\sum_{y=x+1}^{x'}
1\bigl\{\mbox{there is a lozenge of type $\lozv$ at $(y,n)$}\bigr\}.
\end{equation}
Also, for integers $n<n'$ and $y$, set
%
\begin{equation}
\label{versum} \quad\ver_{n,n'}(y):=\sum_{m=n+1}^{n'}
1\bigl\{\mbox{there is a lozenge of type $\lozv$ or $\lozs$ at $(y,m)$}\bigr\}.
\end{equation}
\end{definition}

Above we have explained how each individual fluctuation $H_N(\chi
_i,\eta
_i)$, $i=1,\ldots,\spt$, can be written as a finite linear combination
(with coefficients $\pm1$) of expressions of the form $\hor
_{x,x'}(n)-\E\hor_{x,x'}(n)$ and $\ver_{n,n'}(y)-\E\ver
_{n,n'}(y)$.\footnote{Note that the height function vanishes at the end
of each path in Figure~\ref{figGFFpaths}.} Moreover, if all the
points $(\chi_i,\eta_i)$ are distinct, then each indicator entering any
of the sums~(\ref{horsum}) and (\ref{versum}) corresponding to
$(\chi
_i,\eta_i)$, $i=1,\ldots,\spt$, appears only once because our paths to
the boundary do not intersect.


\subsection{Expectation of a product of horizontal and vertical sums} %
\label{subexpectationofaproductofhorizontalandverticalsums}

Let all the points $(\chi_i,\eta_i)$, $i=1,\ldots,\spt$, be distinct.
From the discussion of Section~\ref{subpathstotheboundary}, it
follows that our expectation (\ref{EHNpathstoboundary}) can be
expressed as a linear combination (with coefficients $\pm1$) of terms
of the form
%
\begin{equation}
\label{horverproduct} \E \Biggl( \prod_{i=1}^{\rpt}
\bigl( \hor_{x_i,x_i'}(n_i)-\E\hor_{x_i,x_i'}(n_i)
\bigr) \prod_{j=\rpt+1}^{\spt} \bigl(
\ver_{n_j,n_j'}(x_j)-\E\ver_{n_j,n_j'}(x_j)
\bigr) \Biggr)
\end{equation}
such that the following horizontal and vertical segments,
\begin{eqnarray*}
&\displaystyle \bigl\{(y,n_i)\dvtx y=x_i+1,\ldots,x_i'
\bigr\},\qquad i=1,\ldots,\rpt;&
\\
&\displaystyle \bigl\{(x_j,m)\dvtx m=n_j+1,\ldots,n_j'
\bigr\},\qquad j=\rpt+1,\ldots,\spt,&
\end{eqnarray*}
do not intersect. Here and below we assume that $x_i<x_i'$, $n_j<n_j'$
for all $i,j$.

\begin{proposition}[{[{cf. \citet{Ferrari2008}}, Lemma 5.3]}]
\label{propsfoldsumofdeterminants}
With the above notation and assumptions, expression (\ref
{horverproduct}) can be written in the following form:
%
\begin{eqnarray}
\label{sfoldsumofdeterminants} \sum_{y_1=x_1+1}^{x_1'}
\cdots \sum_{y_\rpt=x_\rpt+1}^{x_\rpt'} \sum
_{m_{\rpt+1}=n_{\rpt+1}+1}^{n_{\rpt+1}'} \cdots \sum_{m_\spt=n_{\spt}+1}^{n_{\spt}'}
\det %
\left[\matrix{ A_{1,1}&A_{1,2}
\cr
A_{2,1}&A_{2,2}} \right].
\end{eqnarray}
The matrix blocks are given by
%
\begin{eqnarray}\label{A11blocks}
A_{1,1}&=& \bigl[(1-\delta_{ij})K(y_i,n_i;y_j,n_j)
\bigr]_{i,j=1,\ldots,\rpt},
\nonumber\\
A_{1,2}&=&\bigl[K(y_i,n_i;x_j+1,m_j-1)
\bigr]_{i=1,\ldots,\rpt; j=\rpt+1,\ldots
,\spt},
\nonumber
\\[-8pt]
\\[-8pt]
\nonumber
 A_{2,1}&=&\bigl[-K(x_i,m_i;y_j,n_j)
\bigr]_{i=\rpt+1,\ldots,\spt; j=1,\ldots,\rpt,
}
\\
A_{2,2}&=& \bigl[-(1-\delta_{ij})K(x_i,m_i;x_j+1,m_j-1)
\bigr]_{i,j=\rpt+1,\ldots,\spt}.
\nonumber
\end{eqnarray}
\end{proposition}

\begin{pf}
As was observed in [\citet{Kenyon2004Height}, Proof of Theorem 7.2], the
subtraction of the means in (\ref{horverproduct}) leads to vanishing
of the diagonal matrix elements in $A_{1,1}$ and $A_{2,2}$ in (\ref
{A11blocks}). Thus, it suffices to consider $\E (\prod_{i=1}^{\rpt
}\hor_{x_i,x_i'}(n_i)\times  \prod_{j=\rpt+1}^{\spt}\ver
_{n_j,n_j'}(x_j) )$.
We write every $\hor_{x_i,x_i'}(n_i)$ and $\ver_{n_j,n_j'}(x_j)$ as the
corresponding sum (\ref{horsum})--(\ref{versum}). After taking the
expectation, we get an $\spt$-fold sum as in (\ref
{sfoldsumofdeterminants}) with terms
%
\begin{eqnarray}\label{PNtherearelozenges}
&&
\Pp_{\Pc(N)} \bigl(\mbox{There is a lozenge of type $\lozv$ at
$(y_i,n_i)$, $i=1,\ldots,\rpt$;}
\nonumber\\
&&\hspace*{31pt}\qquad \mbox{and of type $\lozv$ or $\lozs$ at $(x_j,m_j)$,
$j=\rpt +1,\ldots,\spt$} \bigr)
\nonumber
\\[-8pt]
\\[-8pt]
\nonumber
&&\qquad= \Pp_{\Pc(N)} \bigl(\mbox{There is a lozenge of type $\lozv$ at
$(y_i,n_i)$, $i=1,\ldots,\rpt$;}
\\
&&\hspace*{41pt}\qquad \mbox{and there is \textit{no} lozenge of type $\lozl$ at $(x_j,m_j)$,
$j=\rpt+1,\ldots,\spt$} \bigr).\hspace*{-15pt}
\nonumber
\end{eqnarray}
The latter probability can be expressed as an $\spt\times\spt$
determinant which has the structure
\begin{eqnarray*}
&&\det %
\left[\matrix{ K^\Delta_{\te}(y_i,n_i,
\lozv;y_{i'},n_{i'},\lozv)\vspace*{2pt}\cr
K^\Delta_{\te}(x_j,m_j,
\lozl;y_{i'},n_{i'},\lozv)}\right.\\
&&\hspace*{20pt}\left.\matrix{
 K^\Delta_{\te}(y_i,n_i,
\lozv;x_{j'},m_{j'},\lozl)\vspace*{2pt}
\cr
 K^\Delta_{\te}(x_j,m_j,
\lozl;x_{j'},m_{j'},\lozl) }\right] %
_{i,i'=1,\ldots,\rpt; j,j'=\rpt+1,\ldots,\spt}.
\end{eqnarray*}
Here $K_{\te}^{\Delta}$ is the kernel which is obtained from $K_{\te}$
(\ref{extendedkernel}) via a particle-hole involution [e.g., see
\citet
{Borodin2000b}, Appendix A.3] at positions $\{(x_j,m_j)\}_{j=\rpt
+1}^{\spt}$. That kernel $K_{\te}^{\Delta}$ looks as follows (see also
Proposition \ref{propKte}):
\begin{eqnarray*}
K^\Delta_{\te}(y_i,n_i,
\lozv;y_{i'},n_{i'},\lozv) &=&K(y_i,n_i;y_{i'},n_{i'});
\\
K^\Delta_{\te}(y_i,n_i,
\lozv;x_{j'},m_{j'},\lozl)&=& K_{\te}(y_i,n_i,
\lozv;x_{j'},m_{j'},\lozl)
\\
&=& K(y_i,n_i;x_{j'}+1,m_{j'}-1);
\\
K^\Delta_{\te}(x_j,m_j,
\lozl;y_{i'},n_{i'},\lozv)&=& - K_{\te}(x_j,m_j,
\lozl;y_{i'},n_{i'},\lozv)\\
&=& - K(x_j,m_j;y_{i'},n_{i'});
\\
K^\Delta_{\te}(x_j,m_j,
\lozl;x_{j'},m_{j'},\lozl) &=&\delta_{j,j'}-K_{\te}(x_j,m_j,
\lozl;x_{j'},m_{j'},\lozl)
\\
&= &\delta_{j,j'}-K(x_j,m_j;x_{j'}+1,m_{j'}-1).
\end{eqnarray*}
Then, setting the diagonal matrix elements to zero, we obtain matrix
blocks (\ref{A11blocks}). This completes the proof.
\end{pf}

\begin{remark}
Representation of Proposition \ref{propsfoldsumofdeterminants} is
not valid if some of the points $(\chi_j,\eta_j)$ coincide because then
we cannot write all the probabilities (\ref{PNtherearelozenges}) as
$\spt\times\spt$ determinants. In this case, in the asymptotic analysis
of multipoint fluctuations (\ref{EHNpathstoboundary}) we employ
Lemma \ref{lemmaN^epsilonbound} below.
\end{remark}



\section{Asymptotics of the kernel} 
\label{secasymptoticsofthekernel}

In this section we investigate asymptotic properties of the correlation
kernel of Theorem \ref{thmK} in various regimes.

\subsection{Action \texorpdfstring{$S(w;\chi,\eta)$}{$S(w;chi,eta)$}} 
\label{subaction}

The polygon $\Pc(N)$ is assumed to be scaled as in Section~\ref{sublimitshape}. We will be interested in asymptotics of the kernel
$K(x_1,n_1;x_2,n_2)$ when the two points $(x_1,n_1)$ and $(x_2,n_2)$
behave as
%
\begin{equation}
\label{chi1chi2regime} \frac{x_j}{N}\to\chi_j,\qquad
\frac{n_j}{N}\to\eta_j,\qquad j=1,2,
\end{equation}
where $(\chi_1,\eta_1)$ and $(\chi_2,\eta_2)$ are two (not necessarily
distinct) global positions inside the limiting polygon $\Pl$; see
Section~\ref{sublimitshape} and especially Figure~\ref{figfrozenboundary}.

\begin{definition}[{[{\citet{Petrov2012}}, Section~7.2]}]\label{defaction}
Define the \emph{action} by
%
\begin{eqnarray}\label{action}
&&S(w;\chi,\eta)\nonumber\\
&&\qquad:=(w-\chi)\ln(w-\chi) -(w-\chi+1-\eta)\ln(w-\chi+1-\eta)
\\
&&\qquad\quad{}
+(1-\eta)\ln(1-\eta)+\sum_{i=1}^{k}
\bigl[(b_i-w)\ln(b_i-w)-(a_i-w)
\ln(a_i-w) \bigr].\nonumber
\end{eqnarray}
Unless otherwise stated, we assume that that the branches of all
logarithms have cuts looking in negative direction along the real line.
Note that the real part $\Re S(w;\chi,\eta)$ is well defined and
continuous for all $w\in\C$.
\end{definition}

Denote also
\[
\Xi(w;\chi,\eta):=\frac{(w-\chi)(w-\chi+1-\eta)}{1-\eta}
\]
and
%
\begin{equation}
\label{Xi12S12}\qquad \Xi_j(w):=\Xi\biggl(w;\frac{x_j}N,
\frac{n_j}{N}\biggr), \qquad S_j(w):=S\biggl(w;\frac{x_j}N,
\frac{n_j}{N}\biggr),\qquad j=1,2.
\end{equation}

\begin{proposition}\label{propKfirstasymptotics}
In regime (\ref{chi1chi2regime}), the kernel $K(x_1,n_1;x_2,n_2)$
of Theorem~\ref{thmK} has the following asymptotics:
%
\begin{eqnarray}
\label{Kfirstasymptotics}
&&K(x_1,n_1;x_2,n_2)\nonumber\\
&&\qquad=
-1_{n_2<n_1} \biggl(1+O\biggl(\frac{1}{N}\biggr) \biggr)
\frac{1}{2\pi\i} \oint_{\Ga(\chi_2-)} \frac{\exp \{N (S_1(z)-S_2(z) ) \}} {
\sqrt{\Xi_1(z)\Xi_2(z)}}\,dz
\nonumber
\\[-8pt]
\\[-8pt]
\nonumber
&&\qquad\quad{} + \biggl(1+O\biggl(\frac{1}{N}\biggr) \biggr)\\
&&\qquad\quad{}\times \frac{1}{(2\pi\i)^{2}}
\oint_{\Ga(\chi_2-)}\,dz \oint_{\mathfrak{c}(\infty)}\,dw \frac{1}{w-z}
\frac{\exp \{N (S_1(w)-S_2(z) ) \}} {
\sqrt{\Xi_1(w)\Xi_2(z)}}.\nonumber
\end{eqnarray}
Here $z$ in both single and double integrals runs over a
counter-clockwise contour which crosses the real line just to the left
of $\chi_2$, and also to the right of $b_k\sim\frac{B_K}{N}$; see~(\ref
{scaleAiBi}). The $w$ contour is counter-clockwise, contains $\Ga
(\chi
_2-)$ (without intersecting it) and is sufficiently large.
\end{proposition}
When $(\chi_1,\eta_1)=(\chi_2,\eta_2)$, this essentially coincides with
\citet{Petrov2012}, Proposition 7.2.

\begin{pf*}{Proof of Proposition \ref{propKfirstasymptotics}}
Let us adapt the double contour integral in formula (\ref{Kformula})
to the asymptotic regime (\ref{chi1chi2regime}) by scaling the
variables as $\tilde z=z/N$, $\tilde w=w/N$ (and then renaming back to $z,w$),
\begin{eqnarray*}
 K(x_1,n_1;x_2,n_2)&=&
-1_{n_2<n_1}1_{x_2\le x_1}\frac{(x_1-x_2+1)_{n_1-n_2-1}}{(n_1-n_2-1)!} +\frac{N(N-n_1)!}{(N-n_2-1)!}
\\
&&{}\times \frac1{(2\pi\i)^{2}} \oint_{\Ga(\chi_2-)}\,dz
\oint_{\mathfrak{c}(\infty)}\,dw \frac{(Nz-x_2+1)_{N-n_2-1}}{(Nw-x_1)_{N-n_1+1}} \frac{1}{w-z}  \\
&&{}\times\prod
_{i=1}^{k}\frac{(A_i+1/2-Nw)_{B_i-A_i}}{(A_i+1/2-Nz)_{B_i-A_i}}.
\end{eqnarray*}
Here $z$ and $w$ run over the corresponding scaled contours, and they
can be chosen independently of $N$.\footnote{We may drag the $z$
contour slightly to the left of $\chi_2$ because the integrand has
zeroes in $z$ which allow that; and also drag it to the right of $b_k$
because the integrand in (\ref{Kformula}) does not have $z$ poles to
the right of $B_k-\frac12$.} These contours coincide with the ones in
the claim (\ref{Kfirstasymptotics}).

Expressing all the Pochhammer symbols in the integrand above through
the Gamma function and applying the Stirling approximation, we may
write for nonreal $z,w$ [see \citet{Petrov2012}, Section~7.2 for more detail],
\begin{eqnarray*}
&&\frac{1}{w-z}\frac{N(N-n_1)!}{(N-n_2-1)!} \frac{(Nz-x_2+1)_{N-n_2-1}}{(Nw-x_1)_{N-n_1+1}} \prod
_{i=1}^{k} \frac{(A_i+1/2-Nw)_{B_i-A_i}}{(A_i+1/2-Nz)_{B_i-A_i}}
\\
&&\qquad = \biggl(1+O\biggl(\frac{1}{N}\biggr) \biggr) \frac{1}{w-z}
\frac{1}{\sqrt{\Xi_1(w)\Xi_2(z)}} \exp \bigl\{N \bigl(S_1(w)-S_2(z)
\bigr) \bigr\}.
\end{eqnarray*}

As for the additional summand, using Lemma 6.2 in \citet{Petrov2012},
we write
\begin{eqnarray}\label{lemma62v1}
&&  -1_{n_2<n_1}1_{x_2\le x_1}\frac{(x_1-x_2+1)_{n_1-n_2-1}}{(n_1-n_2-1)!}
\nonumber
\\[-8pt]
\\[-8pt]
\nonumber
&&\qquad= -1_{n_2<n_1} \frac{(N-n_1)!}{(N-n_2-1)!} \times\frac{1}{2\pi\i}
\oint_{\Ga(x_2)} \frac{(z-x_2+1)_{N-n_2-1}}{(z-x_1)_{N-n_1+1}}\,dz.
\nonumber
\end{eqnarray}
Then, scaling the $z$ variable as above for the double integral
($\tilde z=z/N$), and using the fact that
\[
\frac{N(N-n_1)!}{(N-n_2-1)!} \frac{(Nz-x_2+1)_{N-n_2-1}}{(Nz-x_1)_{N-n_1+1}} = \biggl(1+O\biggl(\frac{1}{N}
\biggr) \biggr) \frac{\exp \{N (S_1(z)-S_2(z) ) \}} {
\sqrt{\Xi_1(z)\Xi_2(z)}},
\]
we complete the proof.
\end{pf*}


\subsection{Critical points of the action} 
\label{subcriticalpointsoftheaction}

Proposition \ref{propKfirstasymptotics} suggests to use the saddle
point (steepest descent) approach [e.g., see \citet{Okounkov2002},
Section~3] to investigate the asymptotics of the correlation kernel
$K(x_1,n_1;x_2,n_2)$. The first step is to understand critical points
of the action, that is, points where $S'(w;\chi,\eta):=\frac
{\partial
}{\partial w}S(w;\chi,\eta)=0$. Let us recall the results about
critical points obtained in \citet{Petrov2012}:

\begin{longlist}[(1)]
\item[(1)] Depending on the global position $(\chi,\eta)$ inside the
limiting polygon $\Pl$, there are either 0 or 1 critical points of the
action in the (open) upper half plane $\Hb$.

\item[(2)] Points $(\chi,\eta)\in\Pl$ for which there exists a nonreal
critical point [denote it by $\om(\chi,\eta)$] constitute the (open)
liquid region $\D\subset\Pl$ where asymptotically one sees all three
types of lozenges; see Remark \ref{rmkbulkreference}.

\item[(3)]As a function of the global position $(\chi,\eta)\in\D$,
$\om
(\chi,\eta)$ satisfies the algebraic equation (\ref
{omegaalgebraicequation}) and a form of the complex Burgers equation
(\ref{omegacomplexBurgersequation}).

\item[(4)] When $(\chi,\eta)\in\D$ approaches the frozen boundary curve
$\partial\D$ (which separates the liquid region from frozen facets),
the critical point $\om(\chi,\eta)\in\Hb$ merges with its complex
conjugate $\omb(\chi,\eta)$. In addition, points $(\chi,\eta)\in
\partial
\D$ that are cusps ($={}$turning points), or points where $\partial\D$ is
tangent to a side of the polygon (see Figure~\ref{figfrozenboundary})
correspond to certain more special types of merging of the critical
points $\om(\chi,\eta)$ and $\omb(\chi,\eta)$, which we do not
need to
address in our analysis.\footnote{See the explanation in the proof of
Lemma \ref{lemmaedge1} that one can choose paths in Figure~\ref{figGFFpaths} away from such more special points.}

Thus, for all $(\chi,\eta)\in\partial\D$, the action $S(w;\chi,\eta)$
has (at least) double critical point $\om(\chi,\eta)\in\R$ which
can be
taken as a real parameter on the frozen boundary curve. The map $\om
^{-1}\dvtx\R\to\partial\D$ is one-to-one and rational; see \citet
{Petrov2012}, Proposition 2.6.
\end{longlist}

\begin{proposition}
The map $\om\dvtx\D\to\Hb$, $(\chi,\eta)\mapsto\om(\chi,\eta
)$, is a
diffeomorphism.
\end{proposition}

\begin{pf}
Finding the image of a point $z\in\Hb$ under the inverse map $\om
^{-1}$ amounts to solving the equation (\ref{omegaalgebraicequation})
%
\begin{equation}
\label{zalgeqndiffeomorphismproof} (z-\chi)\prod_{i=1}^{k}(z-a_i)=
(z-\chi+1-\eta)\prod_{i=1}^{k}(z-b_i)
\end{equation}
for $\chi$ and $\eta$. Since $z\in\Hb$ is complex and $(\chi,\eta)$
must be real, this is actually a pair of real equations. Let us first
rewrite (\ref{zalgeqndiffeomorphismproof}) as
\[
\chi=z \Biggl(1-\prod_{i=1}^{k}
\frac{z-b_i}{z-a_i} \Biggr) +(\chi+\eta-1)\prod_{i=1}^{k}
\frac{z-b_i}{z-a_i}.
\]
Since the imaginary part of $\prod_{i=1}^{k}\frac
{z-b_i}{z-a_i}$ is nonzero for $z\in\Hb$ (Lemma \ref
{lemmaimofproduct} below), one can always solve the equation
\[
\Im\chi=\Im \Biggl(z \Biggl(1-\prod_{i=1}^{k}
\frac
{z-b_i}{z-a_i} \Biggr) +(\chi+\eta-1)\prod_{i=1}^{k}
\frac{z-b_i}{z-a_i} \Biggr)=0
\]
for $\chi+\eta-1$ and thus find a solution $(\chi,\eta)$ which belongs
to $\D$ because we have a bijection on the boundary $\partial\D$. This
implies that the map $\om\dvtx\D\to\Hb$ is bijective.

The map $\om$ is differentiable, and, moreover, the partial
derivatives $\om_\chi$ and $\om_\eta$ cannot both be zero inside
$\D$
because of the complex Burgers equation (\ref
{omegacomplexBurgersequation}). One can also check that the inverse
map is differentiable. This concludes the proof.
\end{pf}

\begin{lemma}\label{lemmaimofproduct}
Let $a_1<b_1<\cdots<a_k<b_k$, $\sum_{i=1}^{k}(b_i-a_i)=1$, be the
parameters of the limiting polygon $\Pl$. Then
\[
\Im \Biggl(\prod_{i=1}^{k}
\frac{z-b_i}{z-a_i} \Biggr)\ne0,\qquad z\in\Hb.
\]
\end{lemma}

\begin{pf}
Observe that the argument of $\frac{z-b_i}{z-a_i}$ is the angle under
which the segment $[a_i,b_i]$ is seen from the point $z$. Thus, the
argument of the whole product $\prod_{i=1}^{k}\frac
{z-b_i}{z-a_i}$ must be strictly between $0$ and $\pi$, and so the
imaginary part of that product cannot vanish.
\end{pf}


\subsection{Moving the contours} 
\label{submovingthecontours}

Our aim in this subsection is to explain how we deform the contours in
the double integral in (\ref{Kfirstasymptotics}) to employ the saddle
point analysis.

Let us assume that (not necessarily distinct) limiting global positions
$(\chi_1,\eta_1)$ and $(\chi_2,\eta_2)$ in (\ref{chi1chi2regime})
belong to the liquid region $\D\subset\Pl$. Denote the corresponding
critical points of the action by $\om_j:=\om (\frac
{x_j}{N},\frac
{n_j}{N} )\in\Hb$, $j=1,2$.

The behavior of $S_{1,2}$ around $\om_{1,2}$ is quadratic because these
critical points are simple; see also the proof of Proposition \ref
{prop4summands}. Thus there are four curves starting from each point
$\om_{1,2}$ along which the imaginary part $\Im(S_{1,2}(w))$ is
constant; see Figure~\ref{figIm}. As the new $w$ contour we choose the
counter-clockwise closed contour passing through $\om_1$ composed of
two curves with $\Im(S_{1}(w)-S_{1}(\om_{1}))=0$ on which $\Re
(S_{1}(w)-S_{1}(\om_{1}))<0$ for $w\ne\om_1,\omb_1$ (Figure~\ref{figIm}, left). The new counter-clockwise $z$ contour must pass
through $\om_2$ and look like the one on Figure~\ref{figIm}, right, so
on it we will have $\Re(S_{2}(z)-S_{2}(\om_{2}))>0$ for $z\ne\om
_2,\omb_2$.
%
\begin{figure}[b]

\includegraphics{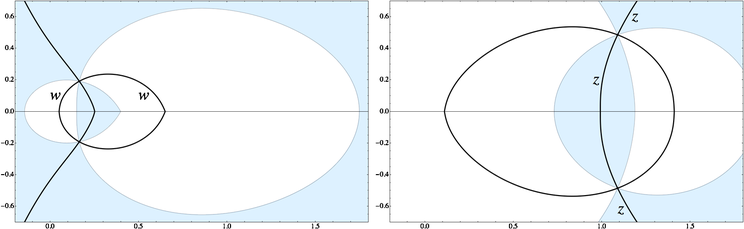}

\caption{Critical points $\om_1$ (left) and $\om_2$ (right). Along the
bold curves one has $\Im(S_{j}(\cdot)-S_{j}(\om_{j}))=0$, $j=1,2$.
Shaded are regions where $\Re(S_{j}(\cdot)-S_{j}(\om_{j}))>0$. The new
$w$ and $z$ contours are indicated on the left and on the right, respectively.}
\label{figIm}
\end{figure}

%
\begin{proposition}\label{propKasymptoticsaftertransform}
The $z$ and $w$ contours in the double integral in (\ref
{Kfirstasymptotics}) can always be deformed to become the new
contours described above (indicated on Figure~\ref{figIm}). This results
in the following asymptotics of the kernel $K$ in the regime~(\ref
{chi1chi2regime}):
%
\begin{eqnarray}
\label{Kasymptoticsaftertransform} &&K(x_1,n_1;x_2,n_2)\nonumber\\
&&\qquad=
\biggl(1+O\biggl(\frac{1}{N}\biggr) \biggr)\frac{1}{2\pi\i}\int
_{\Ga_{\mathrm{single}}} \,dz \frac{\exp \{N (S_1(z)-S_2(z) ) \}} {
\sqrt{\Xi_1(z)\Xi_2(z)}}
\\
&&\qquad\quad{}+ \biggl(1+O\biggl(\frac{1}{N}\biggr) \biggr) \frac{1}{(2\pi\i)^{2}}
\oint_{\{z\}} \oint_{\{w\}} \frac{dz \,dw}{w-z} \frac{\exp \{N (S_1(w)-S_2(z) ) \}} {
\sqrt{\Xi_1(w)\Xi_2(z)}},\nonumber
\end{eqnarray}
where in the double integral $\{z\}$ and $\{w\}$ are the new deformed contours.

The single integral may or may not be present; this depends on whether
the new contours intersect, and also on the inequality between $n_1$
and $n_2$. All these cases can be unified by choosing an appropriate
contour $\Ga_{\mathrm{single}}$; see Section~\ref{subestimatingthesingleintegral}.
\end{proposition}

\begin{pf}
Let us fix any $(\chi,\eta)\in\D$, and set $\om:=\om(\chi,\eta
)$ and
$S(z):=S(z;\chi,\eta)$. As our \textit{first step}, we aim to justify
that the picture of shaded regions where $\Re(S(z)-S(\om))>0$ looks
exactly as in Figure~\ref{figIm}, and also describe the points where
the four contours $\{z\dvtx\Im S(z)=\Im S(\om)\}$ intersect the real line.

Because $\Re S(z;\chi,\eta)\sim\eta\ln|z|$ as $|z|\to\infty$, and
$0<\eta<1$ inside $\D$, far away on Figure~\ref{figIm} we see a shaded
region, that is, where $\Re(S(z)-S(\om))>0$.

Since $S(z)$ is holomorphic everywhere in $\Hb$, along each of the
four contours $\{z\dvtx\Im S(z)=\Im S(\om)\}$ (the thick curves on
Figure~\ref{figIm}) the sign of $\Re(S(z)-S(\om))$ must be constant.
This implies that each thick curve on Figure~\ref{figIm} from $\om$ to
$\omb$ must be completely inside a shaded or nonshaded region.

Now let us look at the function $\Im(S(z)-\Im S(\om))$ for $z\in\R
+\i
\varepsilon$ for fixed small $\varepsilon>0$. Observe that
\[
\Im\bigl( (t+\i\varepsilon)\ln(t+\i\varepsilon)\bigr)= \varepsilon\ln|t+\i\varepsilon|+t
\arg(t+\i\varepsilon) \sim \pi\cdot(t)_-:=\pi\cdot t 1_{t<0},
\]
where $t\in\R$, as $\varepsilon\to0+$. Thus
\[
\frac{1}{\pi}\Im S(t+\i\varepsilon)\sim(t-\chi)_{-}- (t-\chi+1-
\eta)_{-} +\sum_{j=1}^{k}
\bigl[(b_j-t)_{-}-(a_j-t)_{-}
\bigr];
\]
see Figure~\ref{figImS}. Clearly,
$\frac{\partial}{\partial t}\frac{1}{\pi}\Im S(t+\i\varepsilon
)\sim
1_{t\in[\chi+\eta-1,\chi]}-\sum_{j=1}^{k}1_{t\in[a_j,b_j]}$.
Looking at the slopes of the graph in Figure~\ref{figImS}, we see
that the four contours $\{z\dvtx\Im S(z)=\Im S(\om)\}$ can intersect
the real line in at most three points.\footnote{In fact, the case when
there are infinitely many such points (i.e., when a horizontal part of
the graph in Figure~\ref{figImS} is lying at the horizontal
coordinate line) corresponds to $(\chi,\eta)$ belonging to the frozen
boundary, when the action has a real double critical point; see Section~\ref{subcriticalpointsoftheaction}.} Because of the relation
between these contours and the shaded regions in Figure~\ref{figIm}
explained above, there are exactly three such points of intersection:

\begin{itemize}
\item$t^+\in[\chi+\eta-1,\chi]$, where $\Im S(t^+)=\Im S(\om)$ and
$\Re S(t^+)>\Re S(\om)$;
\item$t_l^-<t_r^-$, both belonging to the union of the segments
$[a_j,b_j]$, where\break $\Im S(t^-_{l,r})=\Im S(\om)$ and $\Re
S(t^-_{l,r})<\Re S(\om)$.
\end{itemize}
Moreover, from Figure~\ref{figImS} we see that $t_l^-<t^+<t_r^-$. The
fourth contour $\{z\dvtx\Im S(z)=\Im S(\om)\}$ [with $\Re S(z)>\Re
S(\om)$] runs to infinity; see Figure~\ref{figIm}.

\begin{figure}

\includegraphics{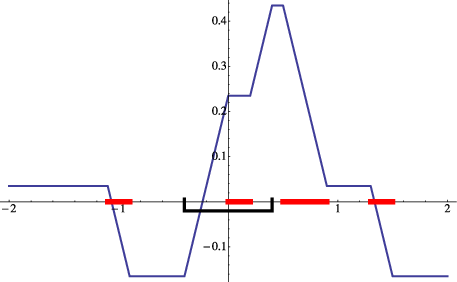}

\caption{Graph of $\frac{1}{\pi}\Im(S(t+\i\varepsilon)-S(\om))$,
$t\in\R
$, for small $\varepsilon>0$. The segments $[a_j,b_j]$ (red) and $[\chi
+\eta-1,\chi]$ (black) are displayed.}
\label{figImS}
\end{figure}

Now as a \textit{second step}, we explain how we can move the $z$ and
$w$ contours in the double contour integral in (\ref
{Kfirstasymptotics}) to get (\ref{Kasymptoticsaftertransform}).
Looking at the poles in $z$ and $w$ in the original integrand in (\ref
{Kformula}), we see that:
\begin{itemize}
\item We can drag the points of intersection of the $z$ contour with
$\R
$ (without picking any residues) everywhere except in regions where the
slope of the graph on Figure~\ref{figImS} is strictly negative.
\item The same goes for the $w$ contour: we cannot drag it through
regions where the slope of the graph on Figure~\ref{figImS} is
strictly positive.
\end{itemize}

The old $z$ and $w$ contours are described in Proposition \ref
{propKfirstasymptotics}; together with what was said above, we see
that these $z$ and $w$ contours can always be deformed in a desired
way. The new $w$ contour will intersect the real line at points
$t^{-}_{l,r}(\chi_1,\eta_1)$; the new $z$ contour---at $t^+(\chi
_2,\eta_2)$. Because the integrand in (\ref{Kformula}) is regular in
$z$ at $z=\infty$, we can let the $z$ contour pass through infinity.

In the course of this deformation no residues coming from poles on $\R
$ will be picked. However, if the new $z$ and $w$ contours intersect,
the residue at $w=z$ will be picked from the $w$ integral, and then
this residue will be integrated in $z$ over an appropriate arc.
Together with the single integral already present in (\ref
{Kfirstasymptotics}), this will lead to appearance of the single
integral in (\ref{Kasymptoticsaftertransform}). We will describe and
investigate it in Section~\ref{subestimatingthesingleintegral} below.
\end{pf}


\subsection{Estimating the single integral} 
\label{subestimatingthesingleintegral}

The goal of this subsection is to asymptotically estimate the single
integral in (\ref{Kasymptoticsaftertransform}). A priori from the
proof of Proposition~\ref{propKasymptoticsaftertransform} we see
that the integral over $\Ga_{\mathrm{single}}$ may look as follows
[we omit the factor $ (1+O(\frac{1}{N}) )$ and the integrand
$\frac{1}{2\pi\i}\frac{\exp\{N(S_1(z)-S_2(z))\}}{\sqrt{\Xi
_1(z)\Xi_2(z)}}\,dz$]:

%
\begin{figure}[b]

\includegraphics{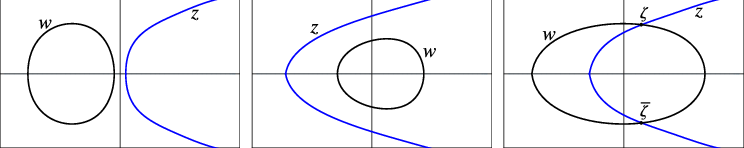}

\caption{Various possibilities for the new deformed $z$ and $w$ contours.}
\label{figzwcontours}
\end{figure}

\begin{longlist}[(a)]
\item[(a)] If the new $z$ and $w$ contours of Proposition \ref
{propKasymptoticsaftertransform} do not intersect, then the
integral has the form $1_{n_2\ge n_1}\oint_{\{z\}}$, where $\{z\}$ is
the full new $z$ contour.
\item[(b)] If the new contours intersect at points $\zeta\in\Hb$ and
$\bar
\zeta$, then the integral has the form $-1_{n_2<n_1}\int_{\Ga
_L(\zeta
)}+1_{n_2\ge n_1}\int_{\Ga_R(\zeta)}$, where $\Ga_L(\zeta)$ is the left
part of the new $z$ contour passed from $\zeta$ to $\bar\zeta$, and
$\Ga
_R(\zeta)$ is its right part passed from $\bar\zeta$ through $\infty$
to $\zeta$.
\end{longlist}
See Figure~\ref{figzwcontours} for the possible configurations of contours.

%
\begin{remark}
If the new $z$ and $w$ contours intersect more than once in $\Hb$,
then the contour $\Ga_{\mathrm{single}}$ would contain several parts.
However, then we always can write an estimate of the form
\[
\int_{\Ga_{\mathrm{single}}}|\cdots| \,dz\le 1_{n_2<n_1}\int
_{\Ga_L(\zeta)}|\cdots| \,dz+ 1_{n_2\ge n_1}\int_{\Ga_R(\zeta')}|
\cdots| \,dz
\]
(dots mean the integrand), where $\zeta$ and $\zeta'$ are some points
of intersection of the new contours. Below (in Lemmas \ref{lemmaCR}
and \ref{lemmaCL} and Proposition \ref
{propsingleintegralestimate}) we estimate the above two summands
separately, so we may think that the case (b) covers all possibilities
when the two contours intersect.
\end{remark}\vspace*{-3pt}

First, we deal with the case (a):\vspace*{-3pt}

%
\begin{lemma}\label{lemmamustintersect}
If the new $z$ and $w$ contours do not intersect, then there is in
fact no single integral in (\ref{Kasymptoticsaftertransform}).
\end{lemma}

\begin{pf}
Recall that the single integral in (\ref
{Kasymptoticsaftertransform}) in that case is asymptotically
equivalent to (see Proposition \ref{propKfirstasymptotics})
\[
1_{n_2\ge n_1}\frac{1}{2\pi\i}\oint_{\{z\}} \frac{N(N-n_1)!}{(N-n_2-1)!}
\frac{(Nz-x_2+1)_{N-n_2-1}}{(Nz-n_1+1)}\,dz.
\]
Here $\{z\}$ is the full new $z$ contour. This integral can be
explicitly computed using Lemma 6.2 in \citet{Petrov2012} [see also
(\ref
{lemma62v1})], it is equal to
\[
1_{n_2\ge n_1}1_{x_2\le x_1}\frac{(n_1-n_2)_{x_1-x_2}}{(x_1-x_2)!}.
\]
This expression is nonzero only if $x_1+n_1\le x_2+n_2$; otherwise the
Pochhammer symbol vanishes. But observe that the three inequalities
\[
n_2\ge n_1,\qquad x_2\le x_1,\qquad
x_1+n_1\le x_2+n_2
\]
in the regime (\ref{chi1chi2regime}) imply that (for large $N$) the
segment $[\chi_2+\eta_2-1,\chi_2]$ is completely inside $[\chi
_1+\eta
_1-1,\chi_1]$. From the proof of Proposition \ref
{propKasymptoticsaftertransform} it follows that the new $z$
contour crosses the real line at some point inside $[\chi_2+\eta
_2-1,\chi_2]$ (and hence inside $[\chi_1+\eta_1-1,\chi_1]$), and the
new $w$ contour passes through two real points at the opposite sides of
$[\chi_1+\eta_1-1,\chi_1]$. Thus, we see that in this situation the new
$z$ and $w$ contours must intersect. This concludes the proof.
\end{pf}

Now we will obtain certain estimates in the case (b).

\begin{lemma}\label{lemmaCR}
Let $n_2\ge n_1$, and $\Ga_R(\zeta)$ for $\zeta\in\Hb$ be defined as
above. We have the estimate
\[
\biggl\llvert \frac{1}{2\pi\i}\int_{\Ga_R(\zeta)}
\frac{\exp\{N(S_1(z)-S_2(z))\}}{\sqrt{\Xi_1(z)\Xi_2(z)}}\,dz\biggr\rrvert \le C \exp\bigl\{N\cdot\Re\bigl(S_1(
\zeta)-S_2(\zeta)\bigr)\bigr\}.
\]
Here the constant $C$ is uniform for $(x_1,n_1),(x_2,n_2)$ in the
regime (\ref{chi1chi2regime}) with the condition $n_2\ge n_1$, for
the limiting global positions $(\chi_1,\eta_1), (\chi_2,\eta_2)$
belonging to a compact region $\D_c\subset\D$.
\end{lemma}

\begin{pf}
Assume first that $(x_1,n_1)\ne(x_2,n_2)$. For large $|z|$, we have
the expansion
\begin{eqnarray*}
 F(z)&:=&S_1(z)-S_2(z)\\
 &=&\mathrm{const}+\frac{n_1-n_2}N
\ln z
\\
&&{}+ \frac{1}{z} \biggl( \frac{n_2^{2}-n_1^{2}}{2N^{2}}+\biggl(1-\frac{x_2}N
\biggr) \biggl(1-\frac{n_2}N\biggr)- \biggl(1-\frac{x_1}N\biggr)
\biggl(1-\frac{n_1}N\biggr) \biggr)\\
&&{} +O\biggl(\frac1{z^{2}}
\biggr).
\end{eqnarray*}

Observe that the function $F(z)$ has no nonreal critical points. This
implies that there is a curve in $\Hb$ \emph{starting} at the point
$\zeta\in\Hb$ along which $\Im F(z)=\Im F(\zeta)$ and $\Re F(z)<\Re
F(\zeta)$ for $z\ne\zeta$. This curve can either extend to infinity, or
cross the real line somewhere in the segment $[\chi_2+\eta_2-1,\chi
_2]$; see Figure~\ref{figCRlemma}. This can be seen by considering
the function $\Im F(t+\i\varepsilon)$ of $t\in\R$ similarly to the proof
of Proposition \ref{propKasymptoticsaftertransform}. Note that for
$n_2=n_1$, such curves will never to go to infinity (Figure~\ref{figCRlemma}, right). In the lower half plane the situation is symmetric.

%
\begin{figure}[b]

\includegraphics{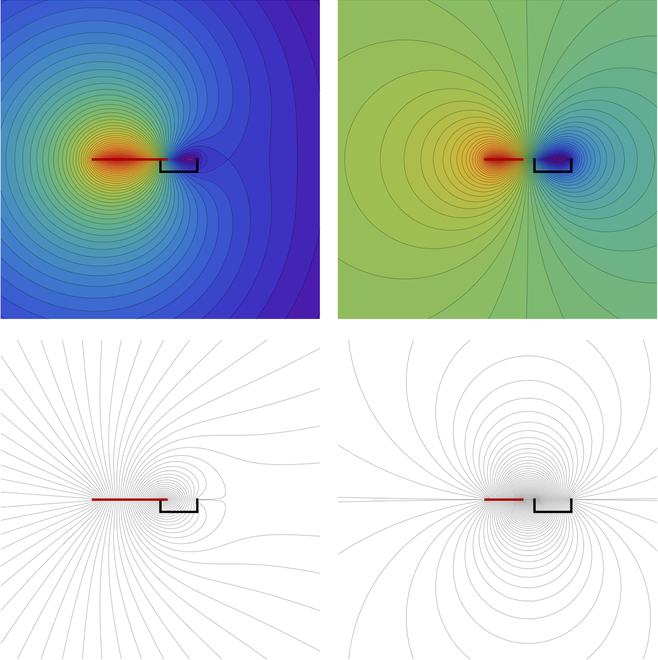}

\caption{Level lines of $\Re F(z)$ (top; warmer colors represent larger
values) and $\Im F(z)$ (bottom) in the case $n_2\ge n_1$ (left) and
$n_2=n_1$ (right). The red (left) segment is $[\chi_1+\eta_1-1,\chi
_1]$, and the black (right) one is $[\chi_2+\eta_2-1,\chi_2]$.}
\label{figCRlemma}
\end{figure}

Since the integrand is regular at $z=\infty$ (see the proof of
Proposition \ref{propKfirstasymptotics}), we can always transform
the contour $\Ga_R(\zeta)$ so that it will consist of curves described
above [along which $\Re F(z)<\Re F(\zeta)$]. This implies the claim for
$(x_1,n_1)\ne(x_2,n_2)$ because if $(\chi_j,\eta_j)\in\D_c$, then the
factor ${1}/{\sqrt{\Xi_1(z)\Xi_2(z)}}$ in the integrand is uniformly
bounded. For $(x_1,n_1)=(x_2,n_2)$, the integral does not depend on
$N$, and the claim is trivial.
\end{pf}

\begin{lemma}\label{lemmaCL}
Let $n_2< n_1$, and $\Ga_L(\zeta)$ for $\zeta\in\Hb$ be defined as
above. We have
\[
\biggl\llvert \frac{1}{2\pi\i}\int_{\Ga_L(\zeta)}
\frac{\exp\{N(S_1(z)-S_2(z))\}}{\sqrt{\Xi_1(z)\Xi_2(z)}}\,dz\biggr\rrvert \le C \exp\bigl\{N\cdot\Re\bigl(S_1(
\zeta)-S_2(\zeta)\bigr)\bigr\}.
\]
Here the constant $C$ is uniform for $(x_1,n_1),(x_2,n_2)$ in the
regime (\ref{chi1chi2regime}) with the condition $n_2< n_1$, for
the limiting global positions $(\chi_1,\eta_1),(\chi_2,\eta_2)$
belonging to a compact region $\D_c\subset\D$.
\end{lemma}

\begin{pf}
This is established in the same way as Lemma \ref{lemmaCR}. Since
$\Ga_L(\zeta)$ never extends to infinity, we can always transform it to
get the desired estimate; see Figure~\ref{figCLlemma}.
\end{pf}

%
\begin{figure}[b]

\includegraphics{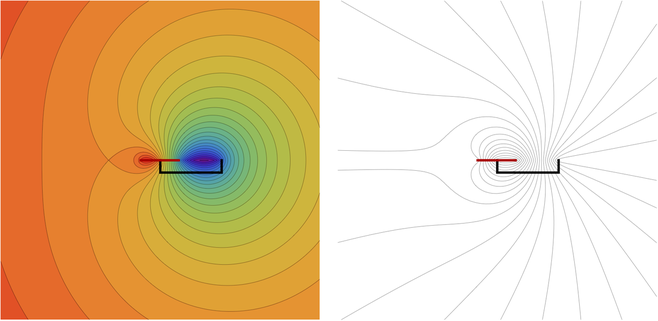}

\caption{Level lines of $\Re F(z)$ (left; warmer colors represent
larger values) and $\Im F(z)$ (right) in the case $n_2< n_1$.}
\label{figCLlemma}
\end{figure}

From Lemmas \ref{lemmaCR} and \ref{lemmaCL} we derive a stronger
estimate which we will use:

%
\begin{proposition}\label{propsingleintegralestimate}
The single integral in (\ref{Kasymptoticsaftertransform}) can be
estimated as
%
\begin{eqnarray}\label{singleintegralestimate}
&& \biggl\llvert \frac{1}{2\pi\i}\int_{\Ga_{\mathrm{single}}}
\frac{\exp\{N(S_1(z)-S_2(z))\}}{\sqrt{\Xi_1(z)\Xi_2(z)}}\,dz\biggr\rrvert
\nonumber\\[-8pt]\\[-8pt]
&&\qquad \le \frac{
C \exp\{N\cdot\Re(S_1(\zeta)-S_2(\zeta))\}}{1+R},
\end{eqnarray}
where $R:=\sqrt{(x_1-x_2)^{2}+(n_1-n_2)^{2}}$. The constant $C$ is
uniform for $(x_1,n_1)$ and $(x_2,n_2)$ behaving as in (\ref
{chi1chi2regime}), for the limiting global positions $(\chi_1,\eta
_1)$, $(\chi_2,\eta_2)$ belonging to a compact region $\D_c\subset\D$.
\end{proposition}

\begin{pf}
The passage from the estimates of Lemmas \ref{lemmaCR} and \ref
{lemmaCL} to (\ref{singleintegralestimate}) is done similarly to
\citet{Duits2011GFF}, Lemma 6.3, and is based on a standard steepest
descent argument.

Let $\zeta\in\Hb$ be the point of intersection of the $z$ and $w$
contours of Proposition \ref{propKasymptoticsaftertransform} (see
also Figure~\ref{figzwcontours}) where the contour $\Ga_{\mathrm
{single}}$ starts. (If the $z$ and $w$ contours do not intersect, the
claim is trivial by Lemma \ref{lemmamustintersect}.) Assume that we
have transformed $\Ga_{\mathrm{single}}$ as in Lemma \ref{lemmaCR} or
\ref{lemmaCL} so that on it we have $\Im F(z)=\Im F(\zeta)$ and
$\Re
F(z)<\Re F(\zeta)$ for $z\ne\zeta$, where $F(z)=S_1(z)-S_2(z)$. If the
new contour extends to infinity (Figure~\ref{figCRlemma}, left), let
us close is so that it will become bounded. Denote this new contour by
$\Ga'_{\mathrm{single}}$.

Let us choose a smooth parametrization $z=z(t)$ of the part of the
curve $\Ga'_{\mathrm{single}}$ inside $\Hb$ such that $z(0)=\zeta$
and $z(1)\in\R$. We have $|z'(t)|<\mathrm{const}$ [and by compactness
of $\D_c$ this constant is uniform in $(x_j,n_j)$], so
\[
\biggl\llvert \frac{1}{2\pi\i}\int_{\Ga_{\mathrm{single}}}
\frac{\exp\{N(S_1(z)-S_2(z))\}}{\sqrt{\Xi_1(z)\Xi_2(z)}}\,dz\biggr\rrvert \le\mathrm{const}\cdot \int_{0}^{1}
\exp\bigl\{N \cdot\Re\bigl(F\bigl(z(t)\bigr)\bigr)\bigr\}\,dt.
\]
Let us assume that $(x_1,n_1)\ne(x_2,n_2)$; otherwise the claim is
again trivial. We have
\[
(x_1-x_2,n_1-n_2)=R(\cos
\phi,\sin\phi),\qquad R>0.
\]
Denote
\[
N \cdot\Re\bigl(F\bigl(z(t)\bigr)\bigr)=R \cdot G_{1,2}(t;\phi).
\]

The property that $\Re F(z)<\Re F(\zeta)$ on our contour allows to
write the following estimate. Since $\D_c$ is compact and $\zeta$
depends continuously on $(\chi_j,\eta_j)$, we can choose $r>0$
uniformly so that
\[
G_{1,2}(t;\phi)-G_{1,2}(0;\phi)<-A t,\qquad 0\le t\le r,
\]
with a constant $A$ not depending on $\zeta$ or $\phi$.

Then we have
\begin{eqnarray*}
\int_{0}^{r}e^{R G_{1,2}(t;\phi)}\,dt&=&
e^{R G_{1,2}(0;\phi)}\int_{0}^{r}e^{R (G_{1,2}(t;\phi
)-G_{1,2}(0;\phi
))}\,dt
\\
&\le& e^{R G_{1,2}(0;\phi)}\int_{0}^{r}e^{-ARt}\,dt
\le \frac{1}{AR}e^{R G_{1,2}(0;\phi)}
\end{eqnarray*}
and
\[
\int_{r}^{1}e^{R G_{1,2}(t;\phi)}\,dt \le
e^{-ArR}e^{R G_{1,2}(0;\phi)},
\]
because on the contour $\{z(t)\dvtx r\le t\le1\}$ we have the
inequalities $\Re F(z)\le\Re F(z(r))\le\Re F(\zeta)-Ar$. This
concludes the proof.
\end{pf}


\subsection{Asymptotics of the kernel in the bulk} 
\label{subasymptoticsofthekernelinthebulk}

Our aim in this subsection is to write an exact asymptotic expansion of
the kernel $K(x_1,n_1;x_2,n_2)$ (\ref{Kasymptoticsaftertransform})
when the two points $(x_1,n_1)$ and $(x_2,n_2)$ are in the bulk of the
system [i.e., they behave as in (\ref{chi1chi2regime}) and $(\chi
_j,\eta_j)\in\D$] and are sufficiently far from each other. The
technique of getting such an expansion involves only ``local''
properties of the double contour integral formula (\ref
{Kasymptoticsaftertransform}) for the kernel (in contrast to some
considerations of Section~\ref{subestimatingthesingleintegral}),
and mainly follows the approach of \citet{Ferrari2008} and \citet{Duits2011GFF}.

\begin{proposition}\label{prop4summands}
Fix (sufficiently small) $\delta>0$ and a compact $\D_c\subset\D$.
Uniformly in $(x_j,n_j)$ ($j=1,2$) with $ (\frac{x_j}{N},\frac
{n_j}{N} )\in\D_c$, such that $\|(x_1,n_1)-(x_2,n_2)\|\ge
N^{{1}/{2}+\delta}$, we have the following expansion:
\begin{eqnarray}\label{Kexpansion}
&&
K(x_1,n_1;x_2,n_2)
\nonumber\\
&&\qquad=- \frac{1}{2\pi N}
\biggl( \frac{e^{N(S_1(\om_1)-S_2(\om_2))}} {
(\om_1-\om_2)\sqrt{\Xi_1(\om_1)\Xi_2(\om_2)}(-S_1''(\om
_1))^{1/2}(S_{2}''(\om_2))^{1/2}}
\nonumber\\
&&\hspace*{38pt}\qquad\quad{} + \frac{e^{N(S_1(\omb_1)-S_2(\om_2))}} {
(\omb_1-\om_2)\sqrt{\Xi_1(\omb_1)\Xi_2(\om_2)}(-S_1''(\omb
_1))^{1/2}(S_{2}''(\om_2))^{1/2}}
\nonumber
\\[-8pt]
\\[-8pt]
\nonumber
&&\hspace*{38pt}\qquad\quad{} + \frac{e^{N(S_1(\om_1)-S_2(\omb_2))}} {
(\om_1-\omb_2)\sqrt{\Xi_1(\om_1)\Xi_2(\omb_2)}(-S_1''(\om
_1))^{1/2}(S_{2}''(\omb_2))^{1/2}}
\nonumber
\\
&&\hspace*{38pt}\qquad\quad{} + \frac{e^{N(S_1(\omb_1)-S_2(\omb_2))}} {
(\omb_1-\omb_2)\sqrt{\Xi_1(\omb_1)\Xi_2(\omb_2)}(-S_1''(\omb
_1))^{1/2}(S_{2}''(\omb_2))^{1/2}} \biggr)\nonumber\\
&&\qquad\quad{}\times \bigl(1+O\bigl(N^{-\delta/2}\bigr)\bigr).
\nonumber
\end{eqnarray}
The branches of the square roots $(\pm S_j''(\cdots))^{1/2}$ above
are chosen in accordance with the directions of the $w$ and $z$
contours in the double integral in (\ref
{Kasymptoticsaftertransform}) at points $\om_1$, $\omb_1$ and
$\om
_2$, $\omb_2$, respectively; see (\ref{Kexpansionnewvariables}) in
the proof.
\end{proposition}
%

%
\begin{pf}
Observe that the contribution from the single integral in (\ref
{Kasymptoticsaftertransform}) given in Proposition \ref
{propsingleintegralestimate} is asymptotically negligible in
comparison to the desired expansion (\ref{Kexpansion}). Thus, it
suffices to consider only the double contour integral in (\ref
{Kasymptoticsaftertransform}),
%
\begin{equation}
\label{Kexpansiondoublecontour} \qquad I_2(x_1,n_1;x_2,n_2):=
\frac{1}{(2\pi\i)^{2}} \oint_{\{z\}} \oint_{\{w\}} \frac{dz \,dw}{w-z}
\frac{\exp \{N (S_1(w)-S_2(z) ) \}} {
\sqrt{\Xi_1(w)\Xi_2(z)}}.
\end{equation}
Recall that the $w$ contour passes through the critical points $\om
_1,\omb_1$, and on it we have $\Im S_1(w)=\Im S_1(\om_1)$ and $\Re
S_1(w)<\Re S_1(\om_1)$ for $w\ne\om_1,\omb_1$. Similarly for the $z$
contour: it passes through $\om_2,\omb_2$, and on it $\Im S_2(z)=\Im
S_2(\om_2)$, and $\Re S_2(z)>\Re S_2(\om_2)$ for $z\ne\om_2,\omb_2$.

The main contributions to (\ref{Kexpansiondoublecontour}) come from
neighborhoods of the critical points, and parts of the contours which
are sufficiently far from them give an exponentially small
contribution. Since there are four pairs of critical points, we get
four summands in (\ref{Kexpansion}). Let us consider only the case of
$(\om_1,\om_2)$, for the other three pairs the situation is analogous.

In small neighborhoods of $\om_1$ and $\om_2$ let us replace the
(curved) $w$ and $z$ contours by the corresponding tangent lines.
Introduce the local coordinates $t,s\in[-N^{{\delta}},N^{\delta}]$
as follows:
%
\begin{equation}
\label{Kexpansionnewvariables} w(t)=\om_1+\frac{t}{\sqrt N(-S_1''(\om_1))^{1/2}}, \qquad z(s)=
\om_2+\frac{s}{\sqrt N(S_2''(\om_2))^{1/2}}.
\end{equation}
Here the branches of the square roots $(-S_1''(\om_1))^{1/2}$ and
$(S_2''(\om_2))^{1/2}$ are chosen so that when $t$ (resp., $s$)
increases, the point $w(t)$ [resp., $z(s)$] passes along the tangent
line to the $w$ (resp., $z$) contour in the direction of that contour.

In these new variables the behavior of the exponents in the double
contour integral is
%
\begin{eqnarray}
\label{KexpansionSasymptotics} \lim_{N\to\infty} N
\bigl(S_1\bigl(w(t)\bigr)-S_1(\om_1)
\bigr)&=&-t^{2}/2,
\nonumber
\\[-8pt]
\\[-8pt]
\nonumber
 \lim_{N\to\infty} N\bigl(S_2
\bigl(z(s)\bigr)-S_2(\om_2)\bigr)&=&s^{2}/2.
\end{eqnarray}
The convergence here is uniform for $t,s\in[-N^{\delta},N^{\delta}]$
and also (by compactness of $\D_c$ and continuity) for our values of
$(x_j,n_j)$. Moreover, at the endpoints $t,s=\pm N^{\delta}$ the
expressions $e^{N(S_1(w(t))-S_1(\om_1))}$ and $e^{N(S_2(z(s))-S_2(\om
_2))}$ are exponentially small. Parts of the contours which are even
farther from the critical points $\om_1,\omb_1$ and $\om_2,\omb_2$ thus
give an exponentially negligible contribution.

This implies that the double contour integral (\ref
{Kexpansiondoublecontour}) picks the following contribution from the
neighborhood of $(\om_1,\om_2)$:
%
\begin{eqnarray}\label{Kexpansionmaincontribution}
&&\frac{1}{(2\pi\i)^{2}} \frac{1}{N(-S_1''(\om_1))^{1/2}(S_2''(\om_2))^{1/2}} \int_{-N^{\delta}}^{N^{\delta}}
\int_{-N^{\delta}}^{N^{\delta}} {\,ds \,dt}
\nonumber
\\[-8pt]
\\[-8pt]
\nonumber
&&\qquad\times\frac{\exp \{N (S_1(w(t))-S_2(z(s)) ) \}} {
({w(t)-z(s)})\sqrt{\Xi_1(w(t))\Xi_2(z(s))}}.
\end{eqnarray}
Let us now get rid of nonexponential terms in the integral above. The
map $\om^{-1}\dvtx\Hb\to\D$ is a diffeomorphism, so there exists a
constant $A>0$ such that
\begin{eqnarray*}
|\om_1-\om_2|\ge A\bigl\|(\chi_1,
\eta_1)-(\chi_2,\eta_2)\bigr\|\ge A
N^{-1/2+\delta}.
\end{eqnarray*}
The second derivatives $S_{1,2}''(\om_{1,2})$ are nonzero inside $\D$
and hence are bounded from below in $\D_c$ (recall that they vanish on
the frozen boundary $\partial\D$). Thus, we may write
\[
\frac{1}{w(t)-z(s)}=\frac{1}{\om_1-\om_2}\bigl(1+O\bigl(N^{-\delta/2}\bigr)
\bigr),
\]
where the constant in $O(N^{-\delta/2})$ does not depend on
$\delta
$ (it depends only on $\D_c$). We can also replace $\Xi_1(w(t))\Xi
_2(z(s))$ by $\Xi_1(\om_1)\Xi_2(\om_2)$. This will affect the
asymptotics by a factor which is less significant than $(1+O(N^{-\delta/2}))$. Thus, we may rewrite (\ref
{Kexpansionmaincontribution}) as
\begin{eqnarray*}
&& -\frac{1}{(2\pi)^2} \frac{e^{N(S_1(\om_1)-S_2(\om_2))}}{N(-S_1''(\om_1))^{1/2}(S_2''(\om_2))^{1/2}
(\om_1-\om_2)\sqrt{\Xi_1(\om_1)\Xi_2(\om_2)}
}
\\
&&\qquad\times \bigl(1+O\bigl(N^{-\delta/2}\bigr)\bigr) \int_{-N^{\delta}}^{N^{\delta}}
\int_{-N^{\delta}}^{N^{\delta}} e^{N (S_1(w(t))-S_1(\om_1)-S_2(z(s))+S_2(\om_2) )}{\,ds \,dt}.
\end{eqnarray*}
Taking $N$ large and using the uniform asymptotics (\ref
{KexpansionSasymptotics}), we see that the above double integral
becomes Gaussian and can be explicitly evaluated. This completes the proof.
\end{pf}

\begin{corollary}\label{corKexpansionshifted}
In the setting of Proposition \ref{prop4summands}, we have the same
expansion for $K(x_1,n_1;x_2+1,n_2-1)$ as in (\ref{Kexpansion}), but
with an extra factor of $\frac{\om_2-{x_2}/N}{1-{n_2}/N}$ in each term
with $\om_2$, and with a factor of $\frac{\omb_2-{x_2}/N}{1-{n_2}/N}$
in each term with $\omb_2$.
\end{corollary}

\begin{pf}
This is obtained in the same way as Proposition \ref{prop4summands}
using the fact that
\[
-NS\biggl(z;\frac{x_2+1}N,\frac{n_2-1}N\biggr)\sim -NS\biggl(z;
\frac{x_2}N,\frac{n_2}N\biggr)-\ln\biggl(1-\frac{n_2}N
\biggr)+\ln\biggl(z-\frac{x_2}N\biggr).
\]
See also Lemma 7.4 in \citet{Petrov2012}.
\end{pf}

We can also write an estimate of the double contour integral $I_2$
(\ref
{Kexpansiondoublecontour}) when the points $(x_1,n_1)$ and
$(x_2,n_2)$ are sufficiently close:

%
\begin{lemma}\label{lemmaI2whentwopointsareclose}
Fix (sufficiently small) $\delta>0$ and a compact $\D_c\subset\D$.
Uniformly in $(x_j,n_j)$ ($j=1,2$) with $ (\frac{x_j}{N},\frac
{n_j}{N} )\in\D_c$, such that $\|(x_1,n_1)-(x_2,n_2)\|\le
N^{{1}/{2}+\delta}$, we have the following estimate:
\[
\bigl|I_2(x_1,n_1;x_2,n_2)\bigr|
\le\frac{Ce^{N\cdot\Re(S_1(\om_1)-S_2(\om
_2))}}{\sqrt N}.
\]
\end{lemma}

\begin{pf}
We argue as in Proposition \ref{prop4summands}, but now we
must estimate the term $1/(w(t)-z(s))$ in a different way. Since the
points $\om_1$ and $\om_2$ are close, we can write $S_1''(\om
_1)=S_2''(\om_2)(1+O(N^{-1/2+\delta}))$, where the constant in
$O(N^{-1/2+\delta})$ is uniform in our $(x_j,n_j)$'s and depends
only on $\D_c$. This implies that in $1/(w(t)-z(s))$ we can replace
$(-S_1''(\om_1))^{1/2}$ with $\pm\i(S_2''(\om_2))^{1/2}$, where
the sign $\pm$ depends on the choice of square roots. Moreover, we have
$|\om_1-\om_2|=O(N^{-1/2+\delta})$, so we can write
\[
\frac{1}{w(t)-z(s)} \approx \frac{\sqrt N} {
\sqrt N(\om_1-\om_2)-(S_2''(\om_2))^{-1/2}(s\pm\i t)}.
\]
Then we proceed as in the proof of Proposition \ref{prop4summands} and
see that the resulting double integral has the following asymptotics
coming from the neighborhood of $(\om_1,\om_2)$:
\begin{eqnarray*}
&& -\frac{1}{(2\pi)^2} \frac{e^{N(S_1(\om_1)-S_2(\om_2))}}{\sqrt{N}
(-S_1''(\om_1))^{1/2}(S_2''(\om_2))^{1/2}
\sqrt{\Xi_1(\om_1)\Xi_2(\om_2)}}
\\
&&\qquad\times \int_{-N^{\delta}}^{N^{\delta}} \int_{-N^{\delta}}^{N^{\delta}}
\frac{e^{N (S_1(w(t))-S_1(\om_1)-S_2(z(s))+S_2(\om_2) )}} {
\sqrt N(\om_1-\om_2)-(S_2''(\om_2))^{-1/2}(s\pm\i t)} {\,ds \,dt}.
\end{eqnarray*}
(For other three pairs of critical points, one can get the same
estimate.) Depending on how close the points $\om_1$ and $\om_2$ in our
regime, we see that the above integral may have a singularity which is
integrable, and (on the other hand) the expression $\sqrt N(\om_1-\om
_2)$ may go to infinity. This implies that we can always bound the
integral by a constant, and thus we arrive at the desired estimate.
\end{pf}

%
\begin{remark}\label{rmkdoublesingleintegralsufficiency}
In Proposition \ref{prop4summands} we see that when the points
$(x_j,n_j)$ ($j=1,2$) are sufficiently far from each other, the main
contribution to $K(x_1,n_1;x_2,n_2)$ (\ref
{Kasymptoticsaftertransform}) comes from the double contour
integral. On the contrary, for sufficiently close points $(x_j,n_j)$,
the single integral in (\ref{Kasymptoticsaftertransform}) is
asymptotically more significant. On the extreme, when asymptotically
the differences $x_1-x_2,n_1-n_2\in\Z$ stabilize, the double integral
in (\ref{Kasymptoticsaftertransform}) vanishes in the limit, while
the single integral gives rise to the incomplete beta kernel; see
\citeauthor{Petrov2012} [(\citeyear{Petrov2012}), Theorem 2 and Proposition 7.9].
\end{remark}


\subsection{Estimates of the kernel close to the edge and in the facet}
\label{subestimatesofthekernelclosetotheedge}

We conclude this section with several estimates for the kernel
$K(x_1,n_1;x_2,n_2)$ (\ref{Kasymptoticsaftertransform}) when one or
both points $(x_j,n_j)$ becomes close to the lower left edge of the
liquid region~$\D$, or outside $\D$ in the lower left facet; see
Figure~\ref{figGFFpaths}. We mainly follow a similar treatment for a
simpler kernel which was performed in \citet{Ferrari2008}, Section~6.

Recall (Section~\ref{subpathstotheboundary}) that we choose the
paths for calculating the height function as in Figure~\ref{figGFFpaths} ending in the lower left facet. Let $\eta=\eta
_{\mathrm{fb}}(\chi)$ be the equation for the corresponding lower
left part
of the frozen boundary. Thus, the liquid region (sufficiently close to
that part of $\partial\D$) is determined by the inequality $\eta
-\eta
_{\mathrm{fb}}(\chi)>0$. We distinguish three regimes for a point
$(x,n)$ (in the pre-limit integer coordinates):
%
\begin{eqnarray}
&\mbox{(inside the liquid region)} \qquad n-N\eta_{\mathrm{fb}} (x/N )\ge
N^{2/3};&\label{xnliquid}
\\
&\mbox{(close to the edge)}\qquad  N^{2/3}\ge n-N\eta_{\mathrm
{fb}} (x/N )
\ge c N^{1/3};&\label{xncloseedge}
\\
&\mbox{(at the edge or in the facet)}\qquad  n-N\eta_{\mathrm{fb}} (x/N )\le c
N^{1/3}&\label{xnedgefacet}
\end{eqnarray}
for some $c>0$.

\begin{lemma}\label{lemmaedge1}
Assume that the points $(x_j,n_j)$ ($j=1,2$) behave as in (\ref
{chi1chi2regime}), and one or both of them is close to the lower
left edge as in (\ref{xncloseedge}). Also, let $|\om_1-\om_2|$ be
bounded away from zero uniformly in $N$. Then there exists $c$ in (\ref
{xncloseedge}) large enough so that we have
%
\begin{equation}
\label{lemmaedge1estimate} \bigl|K(x_1,n_1;x_2,n_2)\bigr|
\le\frac{C e^{N \Re(S_1(\om_1)-S_2(\om_2))}} {
N\sqrt{|S_1''(\om_1)S_2''(\om_2)|}},
\end{equation}
uniformly in $N$.
\end{lemma}

\begin{pf}
Because $|\om_1-\om_2|$ must be bounded from below, we see that the
limiting global positions $(\chi_j,\eta_j)$, $j=1,2$, must be distinct.
Proposition \ref{propsingleintegralestimate} (cf. Remark \ref
{rmkdoublesingleintegralsufficiency}) then implies that the single
integral in (\ref{Kasymptoticsaftertransform}) is asymptotically
less significant than the desired estimate (\ref{lemmaedge1estimate})
for the kernel [note that at least one of the factors $S_1''(\om_1),
S_2''(\om_2)$ goes to zero as $N\to\infty$, see also the proof of Lemma~\ref{lemmaVclosetoedge}]. Therefore, it suffices to derive (\ref
{lemmaedge1estimate}) for the double contour integral in (\ref
{Kasymptoticsaftertransform}).

As usual in the steepest descent approach, the main contribution to
the double contour integral in (\ref{Kasymptoticsaftertransform})
comes from the neighborhoods of the critical points. Thus, there we
have $w\approx\om_1$, $z\approx\om_2$ (plus three more possibilities
with $\omb_{1,2}$ replacing $\om_{1,2}$, but they give the same
contribution to the desired bound).

Let, by agreement, the paths in Figure~\ref{figGFFpaths} be
separated [in the limiting coordinates $(\chi,\eta)$] form the tangent
points of the frozen boundary and the sides of the polygon. Clearly,
such paths still can be chosen. Thus, we may think that the quantities
$\Xi_1(\om_1)$ and $\Xi_2(\om_2)$ are uniformly bounded away from zero;
see also \citet{Petrov2012}, Proposition 2.7 and Figure~14. Thus, it
remains to estimate the product of two single integrals
\[
\oint_{\{w\}} e^{N S_1(w)}\,dw \oint_{\{z\}}
e^{-N S_2(z)}\,dz,
\]
where the $w$ and $z$ contours are as in (\ref
{Kasymptoticsaftertransform}). We will derive the estimate of the
form $\frac{C e^{\pm N \Re S_{1,2}(\om_{1,2})}}{\sqrt {N|S_{1,2}''(\om
_{1,2})|}}$ for each of the single integrals (with ``$+$'' sign for the
first integral, and ``$-$'' for the second one), and this will give the
desired claim. If, say, $(x_1,n_1)$ is not close to the edge, then the
corresponding estimate can be obtained in the same way as in
Proposition \ref{prop4summands}. So let us assume that $(x_1,n_1)$ is
close to the lower left edge, and estimate the $w$ integral above; for
the $z$ integral the argument is the same.

From \citet{Petrov2012}, Proposition 2.7, it follows that for
$(x_1,n_1)$ approaching the lower left edge of the liquid region, the
corresponding critical point $\om_1$ approaches the real line to the
left of the point $a_1$ from (\ref{scaleAiBi}). Using \citet
{Petrov2012}, Section~2.3 and (2.10) (cf. Remark \ref
{rmkbulkreference}), it can be shown that $\arg S''_1(\om_1)$ tends
to $-\pi/2$. Let us introduce the local variable $t$ around~$\om_1$,
\[
w(t)=\om_1+e^{-\i\pi/4}t, \qquad -\delta\le t\le\sqrt2\Im(
\om_1).
\]
It can be readily checked [cf. \citet{Ferrari2008}, Lemma 6.8] that
replacing the $w$ contour around $\om_1$ by the straight line $\{w(t)\}
$ will not affect the desired bound [provided that $c$ in (\ref
{xncloseedge}) is large enough]. We then have
\[
S_1\bigl(w(t)\bigr)=S_1(\om_1)-
\frac{\i}{2}S''_1(
\om_1)t^{2}- \frac{1+\i}{6\sqrt2}S'''_1(
\om_1)t^{3}+O\bigl(t^{4}\bigr),
\]
and
$\Re(-\frac{\i}{2}S''_1(\om_1)t^2)\approx-\frac12|S''_1(\om
_1)|t^2$. Since $\om_1$ is close to the real line, one can derive an
equivalence of the form $S_1'''(\om_1)\approx\frac{S_1''(\om_1)}{\i
\Im
(\om_1)}\approx-\frac{|S_1''(\om_1)|}{\Im(\om_1)}$, and so
\[
\Re\biggl(-\frac{1+\i}{6\sqrt2}S'''_1(
\om_1)t^{3}\biggr) \approx\frac{1}{6\sqrt2\Im(\om_1)}\bigl|S''_1(
\om_1)\bigr|t^3.
\]
We see that for $-\delta\le t\le0$, the cubic term helps the
convergence, and for $0\le t\le\sqrt2\Im(\om_1)$, we can bound the
cubic term by the quadratic term which will also ensure the
convergence. We thus see that the integral of $e^{N (S_1(w)-S_1(\om
_1))}$ around $\om_1$ is equal to a constant times the integral of
$\exp
(-\frac N2|S_1''(\om_1)|t^{2})$, which leads to the desired estimate.
\end{pf}

To describe further estimates, we need to introduce some notation. Let
$(x,n)$ be at the lower left edge or in the corresponding facet as in
(\ref{xnedgefacet}). We would like to mimic the critical point $\om
(\frac xN,\frac nN )$ and the value of the action $S(\om
(\frac xN,\frac nN);\frac xN,\frac nN)$ for such $(x,n)$ as follows:
\begin{eqnarray*}
 \tilde\om&=&\tilde\om \biggl(\frac xN,\frac nN \biggr):=\om \biggl(\frac xN,
\eta_{\mathrm{fb}} \biggl(\frac xN \biggr) \biggr)\in\R,
\\
 \tilde S \biggl( w;\frac xN,\frac nN \biggr)&:=& \biggl(w-\frac{x}N
\biggr) \ln \biggl(w-\frac{x}N \biggr)\\
&&{}- \biggl(w-\frac{x}N+1-
\frac{n}N \biggr) \ln \biggl(w-\frac{x}N+1-\eta_{\mathrm{fb}}
\biggl(\frac{x}N \biggr) \biggr)
\\
&&{}+ \biggl(1-\frac{n}N \biggr) \ln \biggl(1-\eta_{\mathrm{fb}} \biggl(
\frac{x}N \biggr) \biggr)\\
&&{}+ \sum_{i=1}^{k}
\bigl[(b_i-w)\ln(b_i-w)-(a_i-w)
\ln(a_i-w) \bigr].
\end{eqnarray*}
Denote by $\tilde\om_{1,2}$ and $\tilde S_{1,2}(w)$ the corresponding
quantities at $(x_{1,2},n_{1,2})$ similarly to $\om_{1,2}$ and
$S_{1,2}(w)$; see also (\ref{Xi12S12}). Note that when the point
$(x,n)$ is at the edge, that is, when $n-N\eta_{\mathrm{fb}}
(x/N )=O(N^{1/3})$, we have $|\om-\tilde\om|=O(N^{1/3})$,
and same for $\tilde S$.

\begin{lemma}\label{lemmaedge2}
Assume that the situation is as in Lemma \ref{lemmaedge1}, but now
the point $(x_2,n_2)$ is at the lower left edge or in the corresponding
facet as in (\ref{xnedgefacet}), while $(x_1,n_1)$ is in the bulk
(\ref{xnliquid}) or close to the edge (\ref{xncloseedge}). Let
$|\om
_1-\tilde\om_2|$ be bounded away from zero uniformly in $N$. Then, also
uniformly in $N$, we get the following estimate with $C,C_2>0$:
\begin{eqnarray*}
&&\bigl|K(x_1,n_1;x_2,n_2)\bigr|\\
&&\qquad\le
\frac{C e^{N \Re S_1(\om_1)}} {
\sqrt{N|S_1''(\om_1)|}} \times \frac{e^{-N\Re\tilde S_2(\tilde\om_2)}}{N^{1/3}} \exp \biggl\{-C_2
N^{2/3} \biggl(\eta_{\mathrm{fb}} \biggl(\frac{x_2}N \biggr)-
\frac{n_2}N \biggr) \biggr\}.
\end{eqnarray*}
\end{lemma}

\begin{pf}
Similarly to the beginning of the proof of Lemma \ref{lemmaedge1}, we
see that it suffices to estimate the product of two single integrals.
The $w$ integral is bounded as in Lemma \ref{lemmaedge1} (yielding the
first factor in the claim). Thus, it remains to estimate the $z$
integral $\oint_{\{z\}}e^{-N S_2(z)}\,dz$. We will mainly follow the
approach of \citet{Ferrari2008}, Section~6.1 [which is in turn based
on the technique first applied in \citet{Borodin2008TASEPII},
Propositions 15 and 17]. We provide a brief derivation omitting certain
bounds which are done in a way similar to what is performed in \citet
{Ferrari2008}, Section~6.1.

Let us first consider the following scaling of $(x_2,n_2)$:
\[
x_2=[\chi_2 N],\qquad n_2=\bigl[N
\eta_{\mathrm{fb}} (\chi_2 )+ uN^{1/3}\bigr],
\]
where $u\in\R$, and $\chi_2$ is some coordinate such that the line
$\{
\chi\dvtx\chi=\chi_2\}$ intersects the lower left part of the frozen
boundary as in Figure~\ref{figGFFpaths}. Let us expand
\[
-NS\biggl(z;\frac{x_2}N,\frac{n_2}N\biggr) \approx -N S\bigl(z;
\chi_2,\eta_{\mathrm{fb}} (\chi_2 )\bigr)-
uN^{1/3}S_{\eta} \bigl(z;\chi_2,
\eta_{\mathrm{fb}} (\chi_2 )\bigr).
\]
We deform the $z$ contour in $\oint_{\{z\}}e^{-N S_2(z)}\,dz$ so that it
will pass through the real double critical point $\tilde\om_2=\om
(\chi
_2,\eta_{\mathrm{fb}} (\chi_2 ))$. As it usually happens for
Airy-type asymptotics, the main contribution to the integral comes from
an $N^{1/3}$-neighborhood of the double critical point. Let us
introduce the local variable $t$,
\[
z=\tilde\om_2+tN^{-1/3},
\]
and continue the above expansion,
\begin{eqnarray*}
&& -NS\bigl(\tilde\om_2+tN^{-1/3};\chi_2,
\eta_{\mathrm{fb}} (\chi _2 )\bigr)- uN^{1/3}S_{\eta}
\bigl(\tilde\om_2+tN^{-1/3};\chi_2,
\eta_{\mathrm{fb}} (\chi _2 )\bigr)
\\
&&\qquad \approx -NS\bigl(\tilde\om_2;\chi_2,
\eta_{\mathrm{fb}} (\chi_2 )\bigr) -\tfrac16 t^3
S'''\bigl(\tilde\om_2;
\chi_2,\eta_{\mathrm{fb}} (\chi_2 )\bigr)
\\
&&\qquad\quad{}- uN^{1/3} S_{\eta} \bigl(\tilde\om_2;
\chi_2,\eta_{\mathrm{fb}} (\chi_2 )\bigr) -ut
S_{\eta}' \bigl(\tilde\om_2;\chi_2,
\eta_{\mathrm{fb}} (\chi_2 )\bigr).
\end{eqnarray*}
The terms $-\frac16 t^3 S'''(\tilde\om_2;\chi_2,\eta_{\mathrm
{fb}} (\chi_2 ))-ut S_{\eta}'(\tilde\om_2;\chi_2,
\eta_{\mathrm{fb}} (\chi_2 ))$ after the integration in the
neighborhood of the double critical point $\tilde\om_2$ give the Airy
function [similarly to \citet{Ferrari2008}, Lemma 6.1]. The terms
containing $N$ contribute to the factor $e^{-N\Re\tilde S_2(\tilde\om
_2)}$ [after substituting $u=N^{2/3} (\frac{n_2}N-\eta
_{\mathrm{fb}} (\frac{x_2}N ) )$].

For general $(x_2,n_2)$ as in (\ref{xnedgefacet}), the desired
exponential estimate containing $e^{\mathrm{const}\cdot u}$ for the
single integral is obtained along the lines of Lemma 6.2 in \citet
{Ferrari2008}. This completes the proof.
\end{pf}

\begin{lemma}\label{lemmaedge3}
Assume that now both the points $(x_j,n_j)$ ($j=1,2$) are at the lower
left edge or in the corresponding facet as in (\ref{xnedgefacet}).
Let $|\tilde\om_1-\tilde\om_2|$ be bounded away from zero uniformly in
$N$. Then
\begin{eqnarray*}
&&\bigl|K(x_1,n_1;x_2,n_2)\bigr|\\
&&\qquad\le
\frac{C e^{N(\Re\tilde S_1(\tilde\om
_1)-\Re
\tilde S_2(\tilde\om_2))}}{N^{2/3}}
\\
&&\qquad\quad{} \times \exp \biggl\{-C_1 N^{2/3} \biggl(
\eta_{\mathrm{fb}} \biggl(\frac{x_1}N \biggr)- \frac{n_1}N
\biggr)-C_2 N^{2/3} \biggl(\eta_{\mathrm{fb}} \biggl(
\frac{x_2}N \biggr)- \frac{n_2}N \biggr) \biggr\}
\end{eqnarray*}
uniformly in $N$, where $C,C_1,C_2>0$.
\end{lemma}

\begin{pf}
This lemma is obtained similarly to the previous Lemma \ref
{lemmaedge2}, but now we derive exponential estimates for both $w$ and
$z$ integrals.
\end{pf}



\section{Completing the proofs} 
\label{seccompletingtheproofs}

In this section we finish the proof of Theorem \ref
{thmmomentconvergenceintro} (Sections~\ref{subexpandingdeterminantsin}--\ref
{subcontributionfromfixedpointfreeinvolutions}), and then explain
how it leads to Theorem \ref{thmweakconvergenceintro} (Section~\ref{subconvergencetogffproofoftheoremthmweakconvergenceintro}).

\subsection{Expanding determinants in $\mathsf{s}$-fold sums \texorpdfstring{(\protect\ref{sfoldsumofdeterminants})}{(21)}} 
\label{subexpandingdeterminantsin}

Fix pairwise distinct points $(\chi_1,\eta_1),\ldots,(\chi_\spt,\eta
_\spt)$ inside the liquid region $\D$. In Section~\ref{secheightfunctionanditsmultipointfluctuations} we showed that
the expectation $\E (H_N(\chi_1,\eta_1 )\cdots H_N(\chi_\spt,\eta
_\spt) )$ of Theorem \ref{thmmomentconvergenceintro} can be
expressed as a linear combination of expressions like (\ref
{sfoldsumofdeterminants})
%
\begin{eqnarray}
\sum_{y_1=x_1+1}^{x_1'}\cdots \sum
_{y_\rpt=x_\rpt+1}^{x_\rpt'} \sum_{m_{\rpt+1}=n_{\rpt+1}+1}^{n_{\rpt+1}'}
\cdots \sum_{m_\spt=n_{\spt}+1}^{n_{\spt}'} \det %
\left[
\matrix{ A_{1,1}&A_{1,2}\vspace*{2pt}
\cr
A_{2,1}&A_{2,2}
}\right],\label{sfoldsumagain}
\end{eqnarray}
where the matrix blocks are given in (\ref{A11blocks}). Each such
$\spt$-fold sum corresponds to a choice of one linear (horizontal or
vertical) part on every $i$th path starting at the point $(\chi_i,\eta
_i)$, $i=1,\ldots,\spt$; see Figure~\ref{figGFFpaths}. Throughout the
section we assume that these paths on Figure~\ref{figGFFpaths} along
which we calculate the height function are separated from tangent
points as explained in the proof of Lemma \ref{lemmaedge1}.

Let us consider one $\spt$-fold sum as in (\ref{sfoldsumagain}).
Expanding the above $\spt\times\spt$ determinant, we write it as the
sum over permutations $\si\in\mathfrak{S}(\spt)$ of terms each of which
is $\mathop{\mathrm{sgn}}\si$ times the product of matrix elements with
indices $(i,\si_i)$, $i=1,\ldots,\spt$. Express $\si$ as a union of
several disjoint cycles. Since the matrix $
\bigl[{A_{1,1}\atop A_{2,1}}\enskip { A_{1,2} \atop A_{2,2}}\bigr]
$ has zero diagonal entries, all these cycles must have length $\ge2$.
In the next subsection we will show that the contribution of
permutations containing cycles of length $\ge3$ becomes negligible in
the limit as $N\to\infty$.


\subsection{Contribution of permutations with cycles of length \texorpdfstring{$\ge3$}{>=3}}
\label{subcontributionfrompermutationswithcyclesoflengthge3}

Let the permutation $\si\in\mathfrak{S}(\spt)$ contain a cycle of
length $\ell\ge3$. To shorten the notation, we assume that this cycle
is $1\to2\to\cdots\to\ell\to1$. In the expansion of the determinant
in (\ref{sfoldsumagain}) we take the product of the kernels and do a
horizontal (over $y_j=x_j+1,\ldots,x_j'$) or vertical (over
$m_j=n_j+1,\ldots,n_j'$) summation. Let us collect terms corresponding
to a fixed index $i=1,\ldots,\ell$. We will assume that the shifts
$i\pm
1$ are given $\mathrm{mod} \ell$. There are four possible cases we consider:

\begin{longlist}[(V)]
\item[(V)] The summation related to the index $i$ is performed over a
vertical segment: $\{(x_i,m_i)\dvtx m_i=n_i+1,\ldots,n_i'\}$. It can
happen that this vertical segment crosses the lower left frozen
boundary; see Figure~\ref{figGFFpaths}. Thus, we split the summation
into three parts according to (\ref{xnliquid})--(\ref{xnedgefacet}).
\end{longlist}

\begin{longlist}[(V; edge or facet)]
\item[(V; edge or facet)] Summation over $I_1:=\{n_i+1,\ldots,n_i'\}
\cap\{m_i\dvtx m_i\le\break   N \eta_{\mathrm{fb}} (x_i/N )+ c
N^{1/3}\}$. Here $\eta_{\mathrm{fb}}$ is defined in Section~\ref{subestimatesofthekernelclosetotheedge}. We need to consider
%
\begin{equation}
\label{Vedgefacetsum} -\sum_{m_i\in I_1}
K(t_{i-1},u_{i-1};x_i+1,m_i-1)K(x_i,m_i;t_{i+1},u_{i+1}).
\end{equation}
The minus sign is coming from the second factor; see (\ref
{A11blocks}). Here and below, the points $(t_{i\pm1},u_{i\pm1})$
(corresponding to indices $\si_{i}$ and $\si_{i}^{-1}$) are equal to
$(y_j,n_j)$ or $(x_j+1,m_j-1)$ for suitable $j$; see
(\ref{A11blocks}).
\end{longlist}

\begin{lemma}\label{lemmaVedgefacet}
The contribution of the sum (\ref{Vedgefacetsum}) over $I_1$ goes
to zero as \mbox{$N\to\infty$}.\vspace*{-1pt}
\end{lemma}

\begin{pf}
By Lemmas \ref{lemmaedge2} and \ref{lemmaedge3}, we can write
\begin{eqnarray*}
&& \biggl\llvert \sum_{m_i\in I_1} K(t_{i-1},u_{i-1};x_i+1,m_i-1)K(x_i,m_i;t_{i+1},u_{i+1})
\biggr\rrvert
\\
&&\qquad \le \frac{\mathrm{const}}{N^{2/3}}\sum_{m_i\in I_1} \exp \biggl\{-
\mathrm{const} \cdot N^{2/3} \biggl(\eta_{\mathrm{fb}} \biggl(
\frac{x_i}N \biggr)- \frac{m_i}N \biggr) \biggr\}
\\
&&\hspace*{85pt}{} \times \mbox{terms
in $(t_{i\pm1},u_{i\pm1})$}
\\
&&\qquad \le \frac{\mathrm{const}}{N^{1/3}} \times \mbox{terms in $(t_{i\pm1},u_{i\pm1})$}.
\end{eqnarray*}
To get the first estimate above we employ considerations similar to
Corollary~\ref{corKexpansionshifted}; this may change the bound only
by a factor of a constant. The second estimate completes the proof.
\end{pf}

\begin{longlist}[(V; close to edge)]
\item[(V; close to edge)] Summation over $I_2:=\{n_i+1,\ldots,n_i'\}
\cap\{m_i\dvtx c N^{1/3}+N\eta_{\mathrm{fb}} (x_i/N
)\le m_i\le N^{2/3}+N\eta_{\mathrm{fb}} (x_i/N )\}$.
\end{longlist}

\begin{lemma}\label{lemmaVclosetoedge}
The contribution of the sum (\ref{Vedgefacetsum}), where $m_i$ runs
over $I_2$ instead of $I_1$, also goes to zero as $N\to\infty$.
\end{lemma}

\begin{pf}
We have\vspace*{-1pt}
\begin{eqnarray*}
&& \biggl\llvert \sum_{m_i\in I_2} K(t_{i-1},u_{i-1};x_i+1,m_i-1)K(x_i,m_i;t_{i+1},u_{i+1})
\biggr\rrvert
\\[-1pt]
&&\qquad \le \frac{\mathrm{const}}{N}\sum_{m_i\in I_2}
\frac{1}{|S''_i(\om(i))|} \times \mbox{terms in $(t_{i\pm1},u_{i\pm1})$}.
\end{eqnarray*}
Here we have used a shorthand notation $\om(i):=\om(\frac
{x_i}N,\frac
{m_i}N)$, and the same for $S_i$; see also (\ref{Xi12S12}). Also, as
in the previous lemma we use argument similar to Corollary~\ref
{GFFcorrelations}, but this again may only change the constant in the bound.

Let us bound $|S''_i(\om(i))|$ from below. Observe that the third
derivative $S'''(w)$ is bounded away from zero for small $\Im w$, if
$\Re w$ belongs to a bounded interval to the left of the point $a_1\in
\R
$ from (\ref{scaleAiBi}). Indeed, under the map $\om^{-1}\dvtx\Hb
\to
\D$, such $w$'s are close to the lower left edge and are separated from
the tangent points. Thus, we can write $|S''_i(\om(i))|\ge\mathrm
{const}\cdot\Im\om(i)$.

Next, it can be seen that the imaginary part $\Im\om(i)$ can be
bounded from below by $\mathrm{const}\cdot|\om_\eta(\chi,\eta
)|\cdot
|\frac{m_i}{N}-\eta_{\mathrm{fb}}(\frac{x_i}{N})+O(N^{-1/6})|$,
where $(\chi,\eta)\in\D$ is some intermediate point closer to the edge
than $(\frac{x_i}{N},\frac{m_i}{N})$. Instead of $O(N^{-1/6})$ one
could take any correction term which is asymptotically smaller than
$\frac{m_i}{N}-\eta_{\mathrm{fb}}(\frac{x_i}{N})$. We have [see also
\citet{Petrov2012}, Section~7.6]
\[
\om_\eta=\frac{\om-\chi}{\eta-1+(\om-\chi)(\om-\chi+1-\eta
)\Sigma(\om)},
\]
where
\[
\Sigma(\om)=\sum_{i=1}^{k} \biggl(\frac1{
\om-b_i}-\frac1{\om -a_i} \biggr).
\]
Using the formula for the action (\ref{action}), we see that
%
\begin{equation}
\label{ometa} \om_\eta=-\frac{1}{S''(\om(\chi,\eta);\chi,\eta)\cdot(\om
-\chi+1-\eta)}.
\end{equation}

Therefore,
\begin{eqnarray*}
\bigl\llvert S''_i\bigl(\om(i)\bigr)
\bigr\rrvert &\ge&\mathrm{const}\cdot \biggl\llvert \frac{{m_i}/{N}-\eta_{\mathrm{fb}}(
{x_i}/{N})+O(N^{-1/6})}{{S''(\om(\chi,\eta);\chi,\eta)}}\biggr\rrvert
\\
&\ge& \mathrm{const}\cdot N^{{1}/{3}} \biggl\llvert {\frac{m_i}{N}-
\eta_{\mathrm{fb}}\biggl(\frac
{x_i}{N}\biggr)+O\bigl(N^{-1/6}
\bigr)}\biggr\rrvert
\end{eqnarray*}
[we estimate the second derivative in the denominator using (\ref
{xncloseedge})]. We obtain
\begin{eqnarray*}
\frac{\mathrm{const}}{N}\sum_{m_i\in I_2}
\frac{1}{|S''_i(\om(i))|}&\le&
\sum_{m_i\in I_2} \frac{{\mathrm{const}}\cdot N^{-{1}/{3}}}{|{m_i}-N\eta_{\mathrm
{fb}}({x_i}/{N})+O(N^{5/6})|} \\
&\le& \mathrm{const}
\cdot N^{-{1}/{3}}\ln N.
\end{eqnarray*}
This completes the proof.
\end{pf}

It is not hard to see from Lemmas \ref{lemmaVedgefacet} and \ref
{lemmaVclosetoedge} that in the two remaining cases \textup
{(V; bulk)} and \textup{(H)} (see below) we may assume that all the
$\ell$ variables we are summing over (which correspond to vertical and
horizontal summations) always belong to the bulk of the system in the
sense of (\ref{xnliquid}).

\begin{longlist}[(V; bulk)]
\item[(V; bulk)] Summation over $I_3:=\{n_i+1,\ldots,n_i'\}\cap\{
m_i\dvtx m_i\ge N^{2/3}+\break  N\eta_{\mathrm{fb}} (x_i/N
)\}
$. From what was said right above, we may as well take $I_3=\{
n_i+1,\ldots,n_i'\}$. We will investigate the asymptotics of (\ref
{Vedgefacetsum}) where now $m_i$ runs over $I_3$ instead of $I_1$.
\end{longlist}

Similarly to the proof of Lemma \ref{lemmaVclosetoedge}, let us
denote $\om(i)=\om(\frac{x_i}N,\frac{m_i}N)$, $\om(i\pm1):=\om
(\frac
{t_{i\pm1}}N,\frac{u_{i\pm1}}N)$, and same for $S_i,\Xi_i$ [see also
(\ref{Xi12S12})]. Also, let $\be_z(i)$ denote the argument of the
tangent vector to the $z$ contour at the point $\om(i)$ as on
Figure~\ref{figIm}, and analogously for $\be_w(i)$. It can be readily
checked [using the global structure of the $z$ and $w$ contours
(Section~\ref{submovingthecontours})] that
%
\begin{eqnarray}\label{betazwproperties}
\be_z(i)&=&\be_w(i)+\frac\pi2\quad \mbox{and}
\nonumber
\\[-8pt]
\\[-8pt]
\nonumber
\be_z(i)+\be_w(i)+\arg S''_i
\bigl(\om(i)\bigr)&=&\frac{3\pi}2+2\pi q \qquad \mbox{(for some $q\in\Z$)}.
\end{eqnarray}

By the nature of our paths (on Figure~\ref{figGFFpaths}), the points
$(t_{i\pm1},u_{i\pm1})$ are sufficiently far from $(x_i,m_i)$. Thus we
may use Proposition \ref{prop4summands} and Corollary \ref
{corKexpansionshifted} and write (here and below $\delta>0$ is
sufficiently small and fixed)
%
\begin{eqnarray}\label{16summands}
&&  -\sum_{m_i\in I_3}
K(t_{i-1},u_{i-1};x_i+1,m_i-1)K(x_i,m_i;t_{i+1},u_{i+1})
\nonumber\\
&&\qquad =
-\frac{1+O(N^{-\de/2})}{2\pi N\nonumber}\\
&&\qquad\quad{}\times\sum_{m_i\in I_3}
\frac{1}{|S_i''(\om(i))|}
\nonumber\\
&&\hspace*{36pt}\qquad\quad{}\times \biggl\{
\biggl[ \frac{e^{\i\be_z(i)}} {
\om(i-1)-\om(i)}\cdot \frac{e^{\i\be_w(i)}} {
\om({i})-\om(i+1)}
\frac{1}{\Xi_i(\om(i))} \frac{\om(i)-x_i/N}{1-m_i/N}
\nonumber\\
&&\hspace*{57pt}\qquad\quad{}-
\frac{e^{\i\be_z(i)}} {
\om(i-1)-\om(i)}\cdot \frac{e^{-\i\be_w(i)}} {
\omb({i})-\om(i+1)} \frac{e^{-2N\i\cdot\Im S_{i}(\om(i))}}{|\Xi_i(\om(i))|}
\\
&&\hspace*{57pt}\qquad\quad{}\times\frac{\om(i)-x_i/N}{1-m_i/N}-
\frac{e^{-\i\be_z(i)}} {
\om(i-1)-\omb(i)}\cdot \frac{e^{\i\be_w(i)}} {
\om({i})-\om(i+1)}\nonumber\\
&&\hspace*{57pt}\qquad\quad{}\times \frac{e^{2N\i\cdot\Im S_{i}(\om(i))}}{|\Xi_i(\om(i))|}
\frac{\omb(i)-x_i/N}{1-m_i/N}+
\frac{e^{-\i\be_z(i)}} {
\om(i-1)-\omb(i)}\nonumber\\
&&\hspace*{126pt}\qquad\quad{}\times \frac{e^{-\i\be_w(i)}} {
\omb({i})-\om(i+1)} \frac{1}{\Xi_i(\omb(i))}
\frac{\omb(i)-x_i/N}{1-m_i/N} \biggr]
\nonumber
\\
&&\hspace*{180pt}\qquad\quad{} \times \bigl(\mbox{terms in $(t_{i\pm1},u_{i\pm1})$}\bigr)
{}+{}\circlearrowleft \biggr\}.
\nonumber
\end{eqnarray}
Here and below $\circlearrowleft$ denotes all additional terms (there
are 12 of them in the above formula) which are obtained by replacing
$\om(i-1)$ and/or $\om(i+1)$ by the corresponding complex conjugate points.

Arguing analogously to Section~5.3 in \citet{Ferrari2008} [case (a/3)],
one can show that the contribution of the oscillating terms above
containing $e^{\pm2N\i\cdot\Im S_{i}(\om(i))}$ becomes negligible [of
order $O(N^{-{1}/{3}+\varepsilon})$] in the limit. The remaining terms
are smooth and change over distances $m_i\sim N$. Therefore, up to an
error of order $O(N^{-{1}/{3}})$, we can replace the summation over
$m_i$ in (\ref{16summands}) by integration in the scaled variables.
Namely, setting $\mu:=m_i/N$, $\eta:=(n_i+1)/N$, $\eta':=n_i'/N$,
$\chi:=x_i/N$, we can rewrite (\ref{16summands}) as
\begin{eqnarray}
&& \label{8summandsintegral} -\frac{1+O(N^{-\de/2})}{2\pi} \int_{\eta}^{\eta'}\,d
\mu \frac{1}{|S_i''(\om(i))|}
\nonumber\\
&&\qquad
\times \biggl\{ \biggl[ \frac{e^{\i\be_z(i)}} {
\om(i-1)-\om(i)}\cdot \frac{e^{\i\be_w(i)}} {
\om({i})-\om(i+1)}
\frac{1}{\Xi_i(\om(i))} \frac{\om(i)-\chi}{1-\mu}
\nonumber
\\[-8pt]
\\[-8pt]
\nonumber
&&\hspace*{9pt}\qquad\quad{}+
\nonumber
\frac{e^{-\i\be_z(i)}} {
\om(i-1)-\omb(i)}\cdot \frac{e^{-\i\be_w(i)}} {
\omb({i})-\om(i+1)} \frac{1}{\Xi_i(\omb(i))}
\frac{\omb(i)-\chi}{1-\mu} \biggr]
\\
&&\hspace*{130pt}\qquad\quad{} \times
\nonumber
\bigl(\mbox{terms in $(t_{i\pm1},u_{i\pm1})$}
\bigr) {}+{}\circlearrowleft \biggr\}.
\end{eqnarray}

The next step we perform is a change of variables. For the term with
$\om(i)$, we set $\zeta_i^+:=\om(i)=\om(\chi,\mu)$. The integration
path $\Gamma_{i}^{+}$ for $\zeta_i^+$ is from $\om(\chi,\eta)$ to
$\om
(\chi,\eta')$, that is, $\Gamma_{i}^{+}$ is the image of the vertical
line from $\eta$ to $\eta'$ in $\D$ under the map $\om\dvtx\D\to
\Hb$.
Form the results of \citet{Petrov2012}, we have [see also (\ref{ometa})]
\[
\frac{\partial\zeta_i^+}{\partial\mu}= -\frac{1}{S''_i(\om(i))\cdot\Xi_i(\om(i))}\frac{\om(i)-\chi
}{1-\mu}.
\]
Symmetrically, for the term with $\omb(i)$, let $\zeta_i^-:=\omb
(i)=\omb
(\chi,\mu)$, and the integration path $\Gamma_{i}^{-}$ for $\zeta_i^-$
is conjugate to that for $\zeta_i^+$. It can be readily verified [in
particular, using (\ref{betazwproperties})] that the above integral
(\ref{8summandsintegral}) can be rewritten as the sum of two integrals,
%
\begin{eqnarray}\label{Vbulkresult}
 &&\frac{1}{2\pi\i}\sum_{\varepsilon_i=\pm} \varepsilon_i
\int_{\Gamma_i^{\varepsilon_i}} \,d\zeta_i^{\varepsilon_i} \biggl[
\frac{1}{\om(i-1)-\zeta_i^{\varepsilon_i}} \frac1{\zeta_i^{\varepsilon_i}-\om(i+1)}
\nonumber
\\[-8pt]
\\[-8pt]
\nonumber
&&\hspace*{109pt}{}\times
\mbox{terms in $(t_{i\pm1},u_{i\pm1})$} {}+{}\circlearrowleft
\biggr].
\end{eqnarray}
There is an additional minus sign for $\Gamma_i^{-}$ because of a
different phase $-(\be_z(i)+\be_w(i))$ in the second summand in (\ref
{8summandsintegral}). Thus we have established the following fact:

%
\begin{proposition}\label{propVbulk}
The sum (\ref{Vedgefacetsum}), where $m_i$ runs over $I_3$ instead
of $I_1$, is [up to a factor of $1+O(N^{-\de/2})$, where $\de>0$ is
a fixed sufficiently small constant] equal to the sum of two integrals
(\ref{Vbulkresult}).
\end{proposition}

\begin{longlist}[(H)]
\item[(H)] The summation related to the index $i$ is performed over a
horizontal segment, $\{(y_i,n_i)\dvtx y_i=x_i+1,\ldots,x_i'\}$.
\end{longlist}

\begin{proposition}\label{propHbulk}
The horizontal sum has the following asymptotics:
\begin{eqnarray*}
&& \sum_{y_i=x_i+1}^{x_i'}K(t_{i-1},u_{i-1};y_i,n_i)K(y_i,n_i;t_{i+1},u_{i+1})
\\
&&\qquad= \frac{1+O(N^{-\de/2})}{2\pi\i}
\\
&&\quad\qquad{}\times \sum_{\varepsilon_i=\pm} \varepsilon_i \int
_{\Gamma_i^{\varepsilon_i}} \,d\zeta_i^{\varepsilon_i} \biggl[
\frac{1}{\om(i-1)-\zeta_i^{\varepsilon_i}} \frac1{\zeta_i^{\varepsilon_i}-\om(i+1)}\\
&&\hspace*{131pt}{} \times
\mbox{terms in $(t_{i\pm1},u_{i\pm1})$} {}+{}\circlearrowleft
\biggr],
\end{eqnarray*}
where all the notation is as in the previous case, except that now the
path of integration $\Gamma_i^{+}$ connects $\om(\frac
{x_i+1}N,\frac
{n_i}N)$ and $\om(\frac{x_i'}N,\frac{n_i}N)$, and $\Gamma_i^-$ is the
conjugate of $\Gamma_i^+$.
\end{proposition}

\begin{pf}
This is established in the same way as Proposition \ref{propVbulk}.
\end{pf}

After considering the four above cases, we conclude this subsection
with the desired statement about permutations $\si\in\mathfrak
{S}(\spt
)$ having cycles of length $\ell\ge3$:

%
\begin{proposition}\label{propcycle3}
Consider one $\spt$-fold sum (\ref{sfoldsumagain}) and expand the
$\spt\times\spt$ determinant as a sum over permutations $\si\in
\mathfrak
{S}(\spt)$. Then the contribution of those $\si$'s having cycles of
length $\ell\ge3$ goes to zero as $N\to\infty$.
\end{proposition}

\begin{pf}
From Lemmas \ref{lemmaVedgefacet}, \ref{lemmaVclosetoedge} and
Propositions \ref{propVbulk}, \ref{propHbulk} it follows that each
cycle $j_1\to j_2\to\cdots\to j_\ell\to j_1$ in $\si$ asymptotically
produces the following sum of $\ell$-fold integrals:
\begin{eqnarray*}
\frac{1}{(2\pi\i)^{\ell}} \sum_{\varepsilon_{1},\ldots,\varepsilon_{\ell}=\pm}
\varepsilon_{1}\cdots\varepsilon_{\ell} \int
_{\Gamma_1^{\varepsilon_{1}}} \,d\zeta_{1}^{\varepsilon_{1}} \cdots \int
_{\Gamma_\ell^{\varepsilon_{\ell}}} \,d\zeta_{\ell}^{\varepsilon_{\ell}} \prod
_{i=1}^{\ell}\frac{1} {
z_{j_{i}}^{\varepsilon_{j_i}}-
{z_{j_{i+1}}^{\varepsilon_{j_{i+1}}}}}.
\end{eqnarray*}
On the other hand, by \citet{Kenyon2004Height}, Lemma 7.3, we have
\[
\sum_{\mathrm{all\ \ell\mbox{-}cycles\ \tau\in\mathfrak{S}(\ell)}} \prod_{i=1}^{\ell}
\frac{1}{U_{\tau_i}-U_{\tau_{i+1}}}=0,\qquad \ell\ge3.
\]
This concludes the proof.
\end{pf}


\subsection{Contribution of fixed-point-free involutions} 
\label{subcontributionfromfixedpointfreeinvolutions}

In Sections~\ref{subexpandingdeterminantsin}--\ref
{subcontributionfrompermutationswithcyclesoflengthge3} we have
shown that if one expands the determinant in (\ref{sfoldsumagain})
as a sum over permutations $\si\in\mathfrak{S}(\spt)$, then the
contribution of permutations $\si$ which are not fixed-point-free
involutions ($={}$pairings) becomes negligible in the limit. Collecting
all summands of the form (\ref{sfoldsumagain}) corresponding to the
expectation $\E (H_N(\chi_1,\eta_1 )\cdots H_N(\chi_\spt,\eta
_\spt
) )$ [with pairwise distinct positions $(\chi_1,\eta_1),\ldots,\break (\chi
_\spt,\eta_\spt)$], we see that
\begin{eqnarray*}
 &&\E \bigl(H_N(\chi_1,\eta_1 )\cdots
H_N(\chi_\spt,\eta_\spt) \bigr)\\
&&\qquad=\bigl(1+O
\bigl(N^{-{\de}/{2}}\bigr)\bigr)\\
&&\qquad\quad{}\times\sum_{{\mathrm{pairings}\ \si\in\mathfrak
{S}(\spt
)}} \prod
_{i=1}^{\spt/2} \frac{1}{(2\pi\i)^{2}}
\\
&&\hspace*{118pt}{}\times \int_{\omb(\chi_{\si(2i-1)},\eta_{\si(2i-1)})} ^{\om(\chi_{\si(2i-1)},\eta_{\si(2i-1)})} \,d\zeta_{2i-1}
\int_{\omb(\chi_{\si(2i)},\eta_{\si(2i)})} ^{\om(\chi_{\si(2i)},\eta_{\si(2i)})}
\,d\zeta_{2i}\\
&&\hspace*{118pt}{}\times\frac{1}{(\zeta_{2i-1}-\zeta_{2i})^{2}}.
\end{eqnarray*}
Note the additional minus sign coming from the signature of each
transposition. The paths of integration in $\Hb$ from $\omb(\chi
_{\si
(j)},\eta_{\si(j)})$ to $\om(\chi_{\si(j)},\eta_{\si(j)})$ are obtained
by linearity (Section~\ref{subpathstotheboundary}) and by symmetry
of the contours $\Gamma^{\pm}_j$ in Propositions \ref{propVbulk},
\ref
{propHbulk}.

Each integral above can be explicitly evaluated,
\begin{eqnarray*}
&&\frac{1}{(2\pi\i)^{2}} \int_{\omb_1} ^{\om_1} \,d
\zeta_1 \int_{\omb_2} ^{\om_2} \,d
\zeta_2 \frac{1}{(\zeta_1-\zeta_2)^{2}}\\
&&\qquad =-\frac{1}{4\pi^{2}} \ln \biggl(
\frac{(\om_1-\om_2)({\omb_1-\omb_2})} {
(\om_1-\omb_2)(\omb_1-\om_2)} \biggr)= \frac{\G(\om_1,\om_2)}{\pi},
\end{eqnarray*}
where $\G$ is the Green function (\ref{Greenfunction}). With this step
we have completed the proof of Theorem \ref{thmmomentconvergenceintro}.


\subsection{Convergence to $\GFF$: Proof of Theorem \texorpdfstring{\protect\ref{thmweakconvergenceintro}}{1.3}} 
\label{subconvergencetogffproofoftheoremthmweakconvergenceintro}

Our aim now is to prove the weak convergence of $\sqrt\pi H_N(\chi,\eta
)$ (viewed as a generalized function on $\D$) to the $\om$-pullback of
the Gaussian free field $\GFF$ on $\Hb$; see Sections~\ref{subgaussianfreefield}--\ref{subresults}. In order to do that, we
need an additional estimate:

\begin{lemma}\label{lemmaN^epsilonbound}
For any $\varepsilon>0$ and any $\spt$ points $(\chi_1,\eta_1),\ldots,(\chi_\spt,\eta_\spt)\in\D$ (not necessarily pairwise distinct) we
have the bound
\[
\E \bigl(H_N(\chi_1,\eta_1 )\cdots
H_N(\chi_\spt,\eta_\spt) \bigr)= O
\bigl(N^{\varepsilon}\bigr),\qquad  N\to\infty.
\]
\end{lemma}

\begin{pf}
If all the points are distinct, we have a better bound $O(1)$ by
Theorem~\ref{thmmomentconvergenceintro}. Next, assume that, say,
$(\chi_1,\eta_1)=(\chi_2,\eta_2)$. Connect $(\chi_1,\eta_1)$ with the
lower left edge by two paths which are close to each other only in a
neighborhood of $(\chi_1,\eta_1)$. As explained in Section~\ref{secheightfunctionanditsmultipointfluctuations}, we calculate
$(H_N(\chi_1,\eta_1))^{2}$ as a product of sums over these two paths.
Fix small $\de>0$, and in each sum consider separately the $N^{{1}/{2}+\delta}$ terms corresponding to $N^{-{1}/{2}+\delta
}$-neighbor\-hood of $(\chi_1,\eta_1)$. All other terms give a
contribution of order $O(1)$ because they involve points which are far
apart, and so one can argue similarly to Proposition~\ref
{prop4summands} and Theorem \ref{thmmomentconvergenceintro}. On the
other hand, the terms corresponding to close points are estimated using
Proposition~\ref{propsingleintegralestimate} and Lemma \ref
{lemmaI2whentwopointsareclose}. The former gives growth of order
$O(\ln N)^{2}$ coming from the single integral in the correlation
kernel (\ref{Kasymptoticsaftertransform}), and the latter provides a
bound $O(N^{2\delta})$ which comes from the double integral in (\ref
{Kasymptoticsaftertransform}). This completes the proof. See also
the end of Section~7 in \citet{Kenyon2004Height}.
\end{pf}

Now we finish the proof of Theorem \ref{thmweakconvergenceintro}. We
argue similarly to \citet{Ferrari2008}, Section~5.5. It suffices to
establish that [see (\ref{weakconvergenceofonefunction})]
\begin{eqnarray}\label{GFFweakproof1}
&& \lim_{N\to\infty}\pi^{\spt/2}\E \Biggl( \prod
_{i=1}^{\spt} \int_{\D}
\phi_i(\chi_i,\eta_i)H_N(
\chi_i,\eta_i)\,d\chi_i d
\eta_i \Biggr)
\nonumber
\\[-8pt]
\\[-8pt]
\nonumber
&&\qquad=
\nonumber
\E \Biggl( \prod_{i=1}^{\spt}
\int_{\Hb} \phi_i\bigl(\om^{-1}(z_i)
\bigr)J(z_i)\GFF(z_i)|dz_i|^{2}
\Biggr),
\end{eqnarray}
where $\phi_1,\ldots,\phi_{\spt}$ are smooth compactly supported test
functions on $\D$. This convergence of moments implies the weak
convergence (\ref{weakconvergenceofonefunction}) of $\sqrt\pi
\int_{\D}\phi(\chi,\eta)\times  H_N(\chi,\eta)\,d\chi \,d\eta$ to the corresponding
Gaussian random variable\break $\int_{\Hb}\phi(\om^{-1}(z))\times  J(z)\GFF
(z)|dz|^{2}$, as well as an obvious multidimensional analogue of this
fact involving convergence to a Gaussian vector.

The left-hand side of (\ref{GFFweakproof1}) is equal to
\[
\int_{\Hb^{\spt}} \prod_{i=1}^{\spt}|dz_i|^{2}
\bigl(\phi_i\bigl(\om^{-1}(z_i)
\bigr)J(z_i) \bigr) \E \bigl( H_N\bigl(
\om^{-1}(z_1)\bigr) \cdots H_N\bigl(
\om^{-1}(z_\spt)\bigr) \bigr).
\]
We split this integration over $\Hb^{\spt}$ into two parts, one where
the points $z_j$ are sufficiently far apart,
\begin{eqnarray*}
\Hb^{\spt}_{\de}:= \bigl\{ (z_1,
\ldots,z_\spt)\in\Hb^{\spt}\dvtx |z_i-z_j|
\ge N^{-1/2+\delta}, 1\le i<j\le\spt \bigr\},
\end{eqnarray*}
and the remaining part $\Hb^{\spt}\setminus\Hb^{\spt}_{\de}$
where some
of them are close. As usual, $\de>0$ is small and fixed.

For the integration over $\Hb^{\spt}_{\delta}$ we use Propositions
\ref
{propVbulk} and \ref{propHbulk} to write
\begin{eqnarray*}
&& \int_{\Hb^{\spt}_\de} \prod_{i=1}^{\spt}|dz_i|^{2}
\bigl(\phi_i\bigl(\om^{-1}(z_i)
\bigr)J(z_i) \bigr) \E \bigl( H_N\bigl(
\om^{-1}(z_1)\bigr) \cdots H_N\bigl(
\om^{-1}(z_\spt)\bigr) \bigr)
\\
&&\qquad= \int_{\Hb^{\spt}_\de} \prod_{i=1}^{\spt}|dz_i|^{2}
\bigl(\phi_i\bigl(\om^{-1}(z_i)
\bigr)J(z_i) \bigr) \E \bigl( \GFF(z_1) \cdots
\GFF(z_\spt) \bigr)+O\bigl(N^{-\de/2}\bigr).
\end{eqnarray*}
Since the logarithms in $\E (\GFF(z_1)\cdots\GFF(z_\spt
) )$
(see Section~\ref{subgaussianfreefield}) are integrable around zero,
we may replace the above integral over $\Hb^{\spt}_\de$ by the same
integral over~$\Hb^{\spt}$.

The integral over the complement $\Hb^{\spt}\setminus\Hb^{\spt
}_{\delta
}$ is bounded using Lemma \ref{lemmaN^epsilonbound},
\begin{eqnarray*}
& &\Biggl\llvert \int_{\Hb^{\spt}\setminus\Hb^{\spt}_\de} \prod
_{i=1}^{\spt}|dz_i|^{2} \bigl(
\phi_i\bigl(\om^{-1}(z_i)
\bigr)J(z_i) \bigr) \E \bigl( H_N\bigl(
\om^{-1}(z_1)\bigr) \cdots H_N\bigl(
\om^{-1}(z_\spt)\bigr) \bigr) \Biggr\rrvert
\\
&&\qquad\le \mathrm{const}\cdot\bigl(N^{-1/2+\de}\bigr)^{2}N^{\varepsilon},
\end{eqnarray*}
where the constant depends only on our test functions $\phi_j$. This
last estimate implies the desired convergence (\ref{GFFweakproof1}),
and completes the proof of Theorem~\ref{thmweakconvergenceintro}.

\section*{Acknowledgments} 
\label{subacknowledgements}

I would like to thank Alexei Borodin for interest in my work and many
fruitful discussions, and Vadim Gorin, Alexey Bufetov and Jeffrey Kuan
for useful comments.




%
%



\printaddresses

\end{document}